\documentclass[review,onefignum,onetabnum,nohypdvips]{siamart171218}

\usepackage{mathrsfs,amsfonts,amsmath,amssymb,bm,mathtools}
\usepackage{multirow,multicol}
\usepackage{algorithmic,algorithm}
\usepackage{booktabs,lscape,url,tabularx,colortbl}
\usepackage[draft]{changes}
\usepackage{setspace}
\usepackage{subfigure}%
\usepackage{tikz}
\usetikzlibrary{shapes.geometric, arrows}
\usepackage{ulem}
\usepackage[titletoc]{appendix}
\usepackage{etoolbox, xstring, mfirstuc, textcase}

\newtheorem{thm}{Theorem}[section]
\newtheorem{defy}[thm]{Definition}
\newtheorem{remark}[thm]{Remark}

\newcommand{\by}{\bm{y}}
\newcommand{\bz}{\bm{z}}
\newcommand{\bx}{\bm{x}}
\newcommand{\bp}{\bm{p}}

\newcommand{\bk}{\bm{k}}
\newcommand{\bj}{\bm{j}}

\newcommand{\bs}{\bm{s}}

\newcommand{\bM}{\bm{M}}
\newcommand{\bP}{\bm{P}}

\newcommand{\bmm}{\bm{m}}
\newcommand{\bell}{\bm{\ell}}

\newcommand{\blam}{\bm{\lambda}}
\newcommand{\bLam}{\bm{\Lambda}}

\newcommand{\bbR}{\mathbb{R}}
\newcommand{\bbC}{\mathbb{C}}
\newcommand{\bbZ}{\mathbb{Z}}
\newcommand{\bbQ}{\mathbb{Q}}
\newcommand{\bbM}{\mathbb{M}}
\newcommand{\bbN}{\mathbb{N}}

\newcommand{\bbT}{\mathbb{T}}

\newcommand{\calB}{\mathcal{B}}
\newcommand{\calP}{\mathcal{P}}

\newcommand{\calM}{\mathcal{M}}

\newcommand{\hF}{\hat{F}}
\newcommand{\hc}{\hat{c}}

\newcommand{\hf}{\hat{f}}

\newcommand{\tF}{\tilde{F}}
\newcommand{\tf}{\tilde{f}}

\newcommand{\tpsi}{\tilde{\psi}}
\newcommand{\tphi}{\tilde{\phi}}
\newcommand{\tPsi}{\tilde{\Psi}}
\newcommand{\tPhi}{\tilde{\Phi}}

\newcommand\tbbint{{-\mkern -16mu\int}}

\newcommand\dbbint{{-\mkern -19mu\int}}

\newcommand\bbint{
	{\mathchoice{\dbbint}{\tbbint}{\tbbint}{\tbbint}}
}

\newcommand{\Tri}{\mbox{Tri}}
\newcommand{\QP}{\mbox{QP}}

\definecolor{mgray}{rgb}{0.9,0.9,0.9}

\title{Numerical Methods and Analysis of Computing Quasiperiodic Systems
	\thanks{Submitted to ...
		\funding{
		}}
	}

\author{Kai Jiang\thanks{
		Hunan Key Laboratory for Computation and Simulation in Science and Engineering,
		Key Laboratory of Intelligent Computing and Information Processing of Ministry of Education, School of Mathematics and Computational Science, Xiangtan University, Xiangtan, Hunan, China, 411105.
		(\email{kaijiang@xtu.edu.cn}, \email{shifengli@smail.xtu.edu.cn}).}
	\and ShiFeng Li\footnotemark[2]
	\and Pingwen Zhang\thanks{
		School of Mathematics and Statistics, Wuhan University, Wuhan, 430072, School of Mathematical Sciences, Peking University, Beijing, 100871, China.
		(\email{pzhang@pku.edu.cn}). } 
	}

\usepackage{array}
\newcolumntype{I}{!{\vrule width 1.0pt}}
\newlength\savedwidth

\newlength\savewidth

\begin{document}
\maketitle
	
\begin{abstract}
Quasiperiodic systems are important space-filling ordered structures, without decay and translational invariance. How to solve quasiperiodic systems accurately and efficiently is of great challenge.
A useful approach, the projection method (PM) [J. Comput. Phys., $\bm{256}$:~428, 2014], has been proposed to compute quasiperiodic systems. Various studies have demonstrated that the PM is an accurate and efficient method to solve quasiperiodic systems. However, there is a lack of theoretical analysis of PM. 
In this paper, we present a rigorous convergence analysis of the PM by establishing a mathematical framework of quasiperiodic functions and their high-dimensional periodic functions. We also give a theoretical analysis of quasiperiodic spectral method (QSM) based on this framework. Results demonstrate that PM and QSM both have exponential decay, and the QSM (PM) is a generalization of the periodic Fourier spectral (pseudo-spectral) method. Then we analyze the computational complexity of PM and QSM in calculating quasiperiodic systems. The PM can use fast Fourier transform, while the QSM cannot.
Moreover, we investigate the accuracy and efficiency of PM, QSM and periodic approximation method in solving the linear time-dependent quasiperiodic Schr\"{o}dinger equation.
\end{abstract}	
	
\begin{keywords}
Quasiperiodic systems, 
Quasiperiodic spectral method,
Projection method,
Birkhoff's ergodic theorem,
Error estimation,
Time-dependent quasiperiodic Schr\"{o}dinger equation. 
\end{keywords}
	
\begin{AMS}
42A75, 65T40, 68W40, 74S25 
\end{AMS}

\section{Introduction}


Quasiperiodic systems are a natural extension of periodic systems.
The earliest quasiperiodic system can trace back to the study of three-body problem\,\cite{poincare1889problem}. 
Many physical systems can fall into the set of quasiperiodicity, including
periodic systems, incommensurate structures, quasicrystals, many-body problems,
polycrystalline materials, and quasiperiodic quantum systems\,\cite{poincare1889problem, shechtman1984metallic, cao2018unconventional, sutton1992irrational}.
The mathematical study of quasiperiodic orders is a beautiful synthesis
of geometry, analysis, algebra, dynamic system, and number theory\,\cite{baake2013aperiodic, Lubensky1988aperiodicity}.
The theory of quasiperiodic functions, even more general almost periodic functions, has been well developed to study quasiperiodic systems in mathematics\,\cite{bohr1947almost,Besicovitch1954almost,levitan1982almost}. However, how to numerically solve quasiperiodic systems in an accurate and efficient way is still of great challenge.


Generally speaking, quasiperiodic systems, related to irrational numbers, are space-filling ordered structures, without decay nor translational invariance. 
This rises difficulty in numerically computing quasiperiodic systems. 
To study such important systems, several numerical methods have been developed.
A widely used approach, the periodic approximation method (PAM), employs a
periodic function to approximate the quasiperiodic
function\,\cite{jiang2018numerical}. 
The conventional viewpoint is that the approximation error could uniformly decay as the supercell gradually becomes large. 
However, a recent theoretical analysis has demonstrated that the error of PAM may not uniformly decrease as the calculation area increases\,\cite{jiang2022pam}.
The second method is the quasiperiodic spectral method (QSM), which approximates quasiperiodic function by a finite summation of trigonometric polynomials based on the continuous Fourier-Bohr transform\,\cite{jiang2018numerical}, also see \Cref{subsec:QSM}. 
The third approach is the projection method (PM)\,\cite{jiang2014numerical}, based on the fact that the quasiperiodic system can be embedded into a high-dimensional periodic system. Then the PM can accurately calculate the high-dimensional periodic system over a torus in a pseudo-spectral manner. Meanwhile, the PM is efficient due to the availability of fast Fourier transform (FFT). Finally, the PM obtains the quasiperiodic system by choosing a corresponding irrational slice of the high-dimensional torus by the projection matrix. 
Extensive studies have demonstrated that the PM can be used to compute quasiperiodic systems to high precision, including quasicrystals \cite{barkan2014controlled,jiang2015stability}, incommensurate quantum systems \cite{zhou2019plane, li2021numerical,gao2023pythagoras}, topological insulators \cite{wang2022effective}, and grain boundaries \cite{cao2021computing, jiang2022tilt}. However, the PM still has a lack of corresponding theoretical guarantees.                                                                                                                                                                                                                                                                                                                                                                                                                                                                                                                                                                                                                                                                                                                                                                                                                                                                                                                                                                                                                                                                                                                                                                                                                                                                                                                                                                                                                                                                                                                                           


In this work, we present a rigorous theoretical analysis of numerical methods for solving quasiperiodic systems. We establish the relationship between quasiperiodic functions and their corresponding high-dimensional periodic functions based on the idea of PM. These mathematical results provide a theoretical framework to analyze the convergence of PM, as well as QSM. We also present another error analysis framework of QSM without using high-dimensional periodic functions. These theoretical results demonstrate that both PM and QSM have exponential convergence. Moreover, we analyze the computational complexity of PM and QSM in solving quasiperiodic systems. The PM can use FFT by introducing discrete Fourier-Bohr transform, see \Cref{subsec:PM}, while the QSM cannot.
Further analysis reveals that the QSM (PM) is an extension of the periodic Fourier spectral (pseudo-spectral) method. Finally, we investigate the accuracy and efficiency of PM, QSM, and PAM to solving the linear time-dependent quasiperiodic Schr\"{o}dinger equation (TQSE).



\section{Preliminaries}
\label{sec:pre}

Before our analysis, we give some preliminaries on quasiperiodic and periodic functions in this section.

\subsection{Preliminaries of quasiperiodic functions}

Let us recall the definition of the quasiperiodic function\,\cite{levitan1982almost}.
Denote 
\begin{align*}
	\bbM^{d\times n}=\{\bM=(\bmm_1,\cdots,\bmm_n)\in\bbR^{d\times n}: \bmm_1,\cdots,\bmm_n  ~\mbox{are~} \bbQ\mbox{-linearly independent}\},
\end{align*}
and define $\bP\in \bbM^{d\times n}$ as the projection matrix.
\begin{defy}
	\label{def:quasiperiodic}
	A $d$-dimensional function $f(\bx)$ is quasiperiodic
	if there exists a continuous $n$-dimensional periodic function
	$F$ $(n\geq d)$ which satisfies
	$f(\bx)=F(\bP^T \bx)$,
	where $\bP$ is the projection matrix.
\end{defy}
In particular, when $n=d$ and $\bP$ is nonsingular, $f(\bx)$ is periodic.
When $n\rightarrow\infty$, $f$ is almost periodic function\,\cite{bohr1947almost}.
For convenience, $F$ in \Cref{def:quasiperiodic} is called the parent
function of $f$ in the following content. 
$\QP(\bbR^d)$ represents the space of all quasiperiodic functions.
In \Cref{sec:NM and analysis}, we will show that $f$ and $F$ can be uniquely determined by each other when the projection matrix $\bP$ is given.

Let $K_T=\{\bx:\bx\in\bbR^d, ~ \vert x_j\vert \leq T, ~j=1,\cdots,d\}$ be the cube in $\bbR^d$. 
The mean value $\calM\{f(\bx)\}$ of $f\in \QP(\bbR^d)$ is defined as
\begin{align*}
	\calM\{f(\bx)\}=\lim_{T\rightarrow +\infty}
	\frac{1}{(2T)^d}\int_{\bs+K_T} f(\bx)\,d\bx
	:=\bbint  f(\bx)\,d\bx,
\end{align*}
where the limit on the right side exists uniformly for all $\bs\in\bbR^d$. 
An elementary calculation shows
\begin{align}
	\calM\{e^{i\blam^T \bx}e^{-i\bm\beta^T \bx}\}
	=\begin{cases}
		1, ~~\blam=\bm\beta,\\
		0, ~~ \blam\neq \bm\beta.
	\end{cases}
	\label{eq:orthAP}
\end{align}
Correspondingly, the continuous Fourier-Bohr transform of $f(\bx)$ is
\begin{align}
	\hf_{\blam}= \calM\{f(\bx)e^{-i\blam^T \bx}\},
	\label{eq:transform-FC}
\end{align}
where $\blam\in\bbR^d$.
Denote $\bLam=\{\blam: \blam = \bP\bk,~\bk\in \bbZ^n \}$ and
the Fourier series associated with the quasiperiodic function
$f(\bx)$ can be written as
\begin{align}
	f(\bx)\sim\sum_{\bk\in \bbZ^n} \hf_{\blam_{\bk}} e^{i\blam_{\bk}^T \bx},
	\label{eq:Fourierseries}
\end{align}
where $\blam_{\bk}=\bP\bk\in\bLam$ are Fourier exponents and $\hf_{\blam_{\bk}}$ defined in \cref{eq:transform-FC} are Fourier coefficients. To simplify the notation, denote $\hf_{\bk}=\hf_{\blam_{\bk}}$. Let 
\begin{align*}
	\QP_1(\bbR^d)=\Big\{f\in\QP(\bbR^d):~ \sum_{\bk\in \bbZ^n} \vert \hf_{\bk}\vert <+\infty\Big\},
\end{align*}
with norm $\Vert f \Vert_{\mathcal{L}^{\infty}(\bbR^d)}=\sup_{\bx\in \bbR^d} \vert f(\bx) \vert$.



%

In general, the convergence of the Fourier series \cref{eq:Fourierseries}
is a challenging problem, see \cite{levitan1982almost} for some sufficient criteria. The following conclusion presents an important convergence property of quasiperiodic function.
\begin{thm}
	\label{thm:converofAP}
	(\cite{Corduneanu1988almost} Chapter 1.3)
	If the Fourier series of a quasiperiodic function is uniformly convergent, then the sum of the series is the given function.
\end{thm}
If the Fourier series of the quasiperiodic function is absolutely
convergent, it is also uniformly convergent.
Therefore, for $f\in \QP_1(\bbR^d)$, we have
\begin{align*}
	f(\bx)=\sum_{\bk\in \bbZ^n} \hf_{\bk} e^{i\blam_{\bk}^T \bx}.
\end{align*}
As a consequence, we can obtain a subspace $\QP_2(\bbR^d)$ of $\QP(\bbR^d)$
\begin{align*}
	\QP_2(\bbR^d)=\Big \{f\in\QP(\bbR^d): \calM\{\vert f\vert^2 \} <+\infty \Big \}
\end{align*}
equipped with norm
\begin{align}
	\Vert f \Vert_{\mathcal{L}^2(\bbR^d)}^2= \calM \{\vert f\vert^2 \}
	=\sum_{\bk\in \bbZ^n} \vert \hf_{\bk}\vert^2,
	\label{eq:Parseval}
\end{align}
and the inner product $(\cdot , \cdot)_{QP_2(\bbR^d)}$ 
\begin{align*}		
	(f_1, f_2)_{QP_2(\bbR^d)}=\bbint  f_1(\bx)\bar{f}_2(\bx)\,d\bx.
\end{align*}
Equality \cref{eq:Parseval} is the Parseval's identity.
Now we introduce the Hilbert space of quasiperiodic functions. Denote $\vert \bx \vert=\sum_{j=1}^{d}\vert x_j\vert $ with $\forall \bx\in\bbR^d$.
For any $m\in\bbN_0=\{m\in \bbZ:m> 0\}$, the Sobolev space $H^\alpha_{QP}(\bbR^d)$ comprises all quasiperiodic functions with partial derivatives order $\alpha \geq 1$ with respect to the inner product $(\cdot, \cdot)_{\alpha}$
	\begin{align*}
		(f_1, f_2)_{\alpha}=(f_1, f_2)_{QP_2(\bbR^d)}+ \sum_{\vert
			m\vert=\alpha}(\partial^{m}_{\bx} f_1, \partial^{m}_{\bx} f_2)_{QP_2(\bbR^d)},
	\end{align*}
	and endowed with norm
	$\Vert f \Vert_{\alpha}^2=\sum_{\bk\in \bbZ^n}(1+\vert \blam_{\bk}\vert^2)^{\alpha}\vert \hat f_{\bk} \vert^2,$
	and semi-norm
	$\vert f \vert_{\alpha}^2=\sum_{\bk\in \bbZ^n}\vert \blam_{\bk}\vert^{2\alpha}\vert \hat f_{\bk} \vert^2.$


\subsection{Preliminaries of periodic functions}

Let $\bbT^n=(\bbR/2\pi \bbZ)^n$ be the $n$-dimensional torus,
then the Fourier transform of $F(\by)$ defined on $\bbT^n$
	\begin{align}
		\hF_{\bk}
		=\frac{1}{\vert \bbT^n\vert}\int_{\bbT^n}e^{-i\bk^T\by}F(\by)\,d\by,~~\bk\in\bbZ^n,
		\label{eq:rasiedFC}
	\end{align} 
	and
	\begin{align*}
		L^{\infty}(\bbT^n)=\Big\{F(\by): \sum_{\bk\in\bbZ^n} \vert\hF_{\bk} \vert < +\infty\Big\}.
	\end{align*}
Further, denote the Hilbert space on $\bbT^n$
\begin{align*}
	L^2(\bbT^n)=\Big\{F(\by): \langle F,F\rangle < +\infty\Big\},
\end{align*}
equipped with inner product
\begin{align*}
	\langle F_1, F_2\rangle=\frac{1}{|\bbT^n|}\int_{\bbT^n}F_1\bar{F}_2\,d\by.
\end{align*}
For any integer $\alpha\geq 0$, the $\alpha$-derivative Sobolev space on $\bbT^n$ is
\begin{align*} 
	H^\alpha(\bbT^n)=\{F\in L^2(\bbT^n): \Vert F\Vert_{\alpha}<\infty \},
\end{align*}
where  
$
	\Vert F \Vert_\alpha
	=\Big(\sum_{\bk\in\bbZ^n}(1+\|\bk \|_2^{2\alpha})
	\vert \hF_{\bk}\vert^2 \Big)^{1/2},
$
with $\|\bk\|^2_2=\sum_{j=1}^{n} \vert k_{j} \vert^2$.
The semi-norm of $H^\alpha(\bbT^n)$ can be defined as
$
	\vert F \vert_\alpha
	=\Big(\sum_{\bk\in\bbZ^n}\|\bk \|_2^{2\alpha}
	\vert \hF_{\bk}\vert^2 \Big)^{1/2}.
$



\section{Algorithms}
\label{sec:alg}
 

In this paper, our purpose is to establish the theoretical analysis of QSM and PM. In this section, we introduce these algorithms before delving into the numerical analysis.
Moreover, we present the implementation framework of PM by defining the discrete Fourier-Bohr transform of quasiperiodic functions.

For an integer $N\in \bbN_0$ and a given projection matrix $\bP\in\bbM^{d\times n}$, denote
\begin{align*}
	K_N^n=\{\bk=(k_j)_{j=1}^n \in\bbZ^n: \, -N \leq  k_j < N \},
\end{align*}
and
\begin{align}
	\bLam^d_{N}=\{\blam = \bP\bk: \bk\in K_N^n \}\subset \bLam.
	\label{eq:set-trun-Lambda}
\end{align}
Obviously, the order of the set $\bLam^d_{N}$ is $\# (\bLam^d_{N})=(2N)^n$.
The finite dimensional linear subspace of $\QP(\bbR^d)$ is
\begin{align*}
	S_N=\mbox{span} 
	\{ e^{i\blam^T \bx}, ~\bx\in\bbR^d, ~\blam\in\bLam^d_{N}\}.
\end{align*}
We denote $\calP_N: \QP(\bbR^d) \mapsto S_N$ the projection operator. For a quasiperiodic function $f(\bx)\in \QP_1(\bbR^d)$ and its Fourier exponent $\blam_{\bk                                                                                                                                                                                                                                                                                                                                                                                                                                                                                                                                                                                                                                                                                                                                                                                                                                                                                                                                                                                                                                                                                                                                                                                                                                                                                                                                                                                                                                                                                                                                                                                                                                                                                                                                                                                                                                                                                                                                                                                                                                                                                                                                                                                                                                                                                                                                                                                                                                                                                                                                                                                                                                                                                                                                                                                                                                                                                                                                                                                                                                                                                                                                                                                                                                                                                                                                                                                                                                                                                                                                                                                                                                                                                                                                                                                                                                                                                                                                                                                                                                                                                                                                                                                                                                                                                                                                                                                                                                                                                                                                                                                                                                                                                                                                                                                                                                                                                                                                                                                                                                                                                                                                                                                                                                                                                                                                                                                                                                                                                                                                                                                                                                                                                                                                                                                                                                                                                                                                                                                                                                                                                                                                                                                                                                                                                                                                                                                                                                                                                                                                                                                                                                                                                                                                                                                                                                                                                                                                                                                                                                                                                                                                                                                                                                                                                                                              }\in \bLam$, we can split it into two parts
\begin{align}
	f(\bx) =\sum_{\bk \in K_N^n} \hf_{\bk} e^{i \blam_{\bk}^T \bx} +\sum_{\bk \in \bbZ^n/K_N^n} \hf_{\bk} e^{i  \blam_{\bk}^T \bx}
	=\calP_N f+(f-\calP_N f).
	\label{def:f}
\end{align}

Next, we present QSM and PM, respectively.
\subsection{Quasiperiodic spectral method (QSM)}
\label{subsec:QSM}

The QSM directly approximates quasiperiodic function $f$ by $\calP_N f$,  
\begin{align*}
f(\bx)\approx	
\calP_N f(\bx) =\sum_{\bk \in K_N^n} \hf_{\bk} e^{i \blam_{\bk}^T \bx},~~~ \bx \in \bbR^d,
\end{align*}
where the quasiperiodic Fourier coefficient $\hf_{\bk}$ is obtained by the continuous Fourier-Bohr transform \cref{eq:transform-FC}.
We will give the error analysis of QSM in \Cref{sec:error-QSM}, and describe the numerical implementation of solving quasiperiodic system in \Cref{subsubsec:QSMimplement}. Note that quasiperiodic Fourier coefficients in QSM are obtained through the continuous Fourier-Bohr transform \cref{eq:transform-FC}, resulting in the QSM cannot use FFT. A further computational complexity analysis will be presented in \Cref{subsubsec:QSMimplement}.


\subsection{Projection method (PM)}
\label{subsec:PM}

The PM embeds the quasiperiodic function $f(\bx)$ into a high-dimensional parent function $F(\by)$, then directly replace the discrete quasiperiodic Fourier coefficients by the discrete parent Fourier coefficients\,\cite{jiang2018numerical, jiang2014numerical}.
We can use the periodic Fourier spectral method to obtain the parent Fourier coefficients. Concretely, we first discretize the tours $\bbT^n$.
Without loss of generality, we consider a fundamental domain $[0,2\pi)^n$ and assume the discrete nodes in each dimension are
the same, \textit{i.e.}, $N_1=N_2=\cdots=N_n=2N$, $N\in \bbN_0$. 
The spatial discrete size $h=\pi/N$. 
The spatial variables are evaluated on the standard
numerical grid $\bbT^n_N$ with grid points $\by_{\bj} =(y_{1,j_1},
y_{2,j_2},\dots, y_{n,j_n})$, $y_{1,j_1}=j_1 h$, $y_{2,j_2}=j_2 h, \dots,
y_{n,j_n}=j_n h$, $0\leq j_1,j_2,\dots, j_n < 2N$. We define the grid function space
\begin{align*}     
\mathcal{G}_N := \{F: \bbZ^n \mapsto \bbC: ~ F~ \mbox{is}~ \bbT^n_N \mbox{-periodic} \}.
\end{align*}
Given any periodic grid functions $F,\, G\in\mathcal{G}_N$, the $\ell^2$-inner product
is defined as 
\begin{align*}
\langle F, G \rangle_N = \frac{1}{(4\pi N)^n}\sum_{\by_{\bj}\in\bbT^n_N}
F(\by_{\bj})\overline{G}(\by_{\bj}).
\end{align*}
For $\bk,\,\bell\in\bbZ^n$, we have the discrete orthogonality condition
\begin{align}
\langle e^{i\bk^T \by_{\bj} }, e^{i\bell^T \by_{\bj} } \rangle_N =
	\begin{cases}
1,~~\bk=\bell + 2N\bmm,~\bmm \in \bbZ^n,\\
0,~~\mbox{otherwise}.
\end{cases}
\label{eq:disOrth}
\end{align}
The discrete Fourier coefficient of $F\in \mathcal{G}_N$ is
\begin{align}
	\tF_{\bk} = \langle F, e^{i \bk^T \by_{\bj}} \rangle_N,~~~
	\bk \in K_N^n.
	\label{eq:parentdisFourierCoeff}
\end{align}
The PM directly takes $\tf_{\bk}=\tF_{\bk}$. We define the \textit{discrete Fourier-Bohr transform} of quasiperiodic function $f(\bm x)$ as
\begin{align}
	f(\bx_{\bj})=\sum_{\blam_{\bk}\in \bLam_N^d} \tilde{f}_{\bk}e^{i\blam_{\bk}^T \bx_{\bj}},
\label{eq:disFourierCoeff}
\end{align}
where \textit{collocation points}  $\bx_{\bj}=\bP\by_{\bj}$, $\by_{\bj}\in \bbT^n_N$.
The trigonometric interpolation of quasiperiodic function is
\begin{align}
	I_N f(\bx)=\sum_{\blam_{\bk}\in \bLam_N^d} \tf_{\bk}e^{i\blam_{\bk}^T\bx}.
	\label{eq:tri-quasi}
\end{align}
Consequently, $I_N f(\bx_j)=f(\bx_j)$.
From the implementation, the PM can use the $n$-dimensional FFT to obtain quasiperiodic Fourier coefficients by introducing the discrete Fourier-Bohr transform \cref{eq:disFourierCoeff}. The concrete computational complexity of PM for solving quasiperiodic systems will be shown in \Cref{subsubsec:PMimplement}.


%

\begin{remark} \label{rmk:methodRelation}
	From the above description, QSM and PM are generalization of the Fourier
	spectral method and Fourier pseudo-spectral method, respectively. 
	When $f(\bm x)$ is periodic, \textit{i.e.,} $n=d$ and the projection matrix $\bP\in \bbM^{d\times d}$ is nonsingular, the QSM (PM) reduces to the periodic Fourier	spectral (pseudo-spectral) method.
\end{remark}

\section{Theoretical framework}
\label{sec:NM and analysis}

From the implementation framework of PM presented in \Cref{subsec:PM}, exploring the relationship between quasiperiodic functions and their parent functions is a prerequisite for its convergence analysis.
Here, we prove that the quasiperiodic Fourier coefficients $\hf_{\bk}$ of \cref{eq:transform-FC} are equal to their parent Fourier coefficients $\hF_{\bk}$ of \cref{eq:rasiedFC}.
\begin{thm}
\label{thm:object}
For a given quasiperiodic function
\begin{align*}
	f(\bx)=F(\bP^T \bx), ~~~\bx\in\bbR^d,
\end{align*}
where $F(\by)$ is its parent function defined on the tours $\bbT^n$ and $\bP$ is the projection matrix, we have
\begin{align}
	\hf_{\bk}=\hF_{\bk},~~\bk\in\bbZ^n,
	\label{eq:object0}
\end{align}
where $\hf_{\bk}$ and $\hF_{\bk}$ are defined by \cref{eq:transform-FC} and \cref{eq:rasiedFC}, respectively.
\end{thm}

We will prove \cref{thm:object} based on the Birkhoff's ergodic theorem\,\cite{walter1982ergodic, pitt1942generalizations}.
Let us start with some basic definitions.
Let $\Omega$ be a set. A $\sigma$-algebra of $\Omega$ is a collection $\calB$ of subsets of $\Omega$ satisfying the following three conditions: (i) $\Omega\in \calB$; (ii) if $B\in \calB$, then $\Omega\backslash B\in \calB$; (iii) if $B_n\in\calB$ for $n\geq 1$, then $\bigcup^{\infty}_{n=1} B_n\in \calB$.
We call the pair $(\Omega, \calB)$ a measurable space. The Lebesgue measure on $(\Omega,
\calB)$ is a function $\mu: \calB\mapsto \bbR^+$ satisfying $\mu(\emptyset)=0$
and $\mu(\bigcup^{\infty}_{n=1} B_n)=\sum^{\infty}_{n=1} \mu(B_n)$ whenever
$\{B_n\}^{\infty}_{n=1}$ is a sequence of members of $\calB$ which are pairwise
disjoint subsets of $\Omega$. A finite measure space is a triple $(\Omega, \calB,
\mu)$ where $(\Omega, \calB)$ is a measurable space and $\mu$ is a finite measure on
$(\Omega, \calB)$. We say $(\Omega, \calB, \mu)$ is a probability space, or a
normalized measure space if $\mu(\Omega)=1$.
\begin{defy}
Suppose that $(\Omega_1, \calB_1, \mu_1)$ and $(\Omega_2, \calB_2, \mu_2)$ are probability spaces.

(i) A transformation $\phi:\Omega_1\mapsto \Omega_2$ is a measure if $\phi^{-1}(\calB_2)\subset \calB_1$;

(ii) A transformation $\phi:\Omega_1\mapsto \Omega_2$ is measure-preserving if $\phi$ is measureable and $\mu_1(\phi^{-1}B_2)=\mu_2(B_2)$, for each $B_2\in\calB_2$.
\end{defy}

\begin{defy}
Let $(\Omega, \calB, \mu)$ be a probability space. A measure-preserving transformation $\phi:\Omega\mapsto \Omega$ is called ergodic if the only member $B\in\calB$ with $\phi^{-1}B=B$ satisfying $\mu(B)=1$ or $\mu(B)=0$.
\end{defy}

\cref{lem:ergodic-condition} gives an equivalent condition of ergodicity.

\begin{lemma}
(Theorem 1.6 \cite{walter1982ergodic})
\label{lem:ergodic-condition}
Let $(\Omega, \calB, \mu)$ be a probability space and $\phi:\Omega\mapsto \Omega$ be measure-preserving mapping, then the following statements are equivalent

(i) $\phi$ is ergodic;

(ii) If $f$ is measurable and $f \circ \phi=f$ a.e., then $f$ is constant a.e.	
\end{lemma}

The high-dimensional Birkhoff's ergodic theorem reads 

\begin{thm}
\label{thm:Birkhoff}
(\cite{pitt1942generalizations})
Let $f(\bz): \Omega \mapsto \bbC$ be integrable
and let the measure-preserving transformation $\phi^{\bx}$, $\bx\in \bbR^d$, satisfies
\begin{align*}
	\phi^{\bm{0}} \bz=\bz,~~\phi^{\bx_1+\bx_2} \bz=\phi^{\bx_1}(\phi^{\bx_2} \bz),~~\bz\in \Omega,
\end{align*}
for any $\bx_1, \bx_2 \in \bbR^d$. Then
\begin{align*}
	\bbint f(\phi^{\bx}\bz) \,d\bx=f^*(\bz)
\end{align*}
exists for almost all $\bz$ in $\Omega$. Moreover,
\begin{align*}
	\int_E f^*(\bz)\, d\bz=\int_E f(\bz)\, d\bz,
\end{align*}
where $E\subset \Omega$ is the invariant subset under $\phi^{\bm x}$.
\end{thm}

In this work, $\Omega=E=\bbT^n$. Given a projection matrix $\bm P=(\bm p_1, \bm p_2, ..., \bm p_n)\in\bbM^{d\times n}$, denote the parameterized translation
\begin{align}
\phi^{\bx}_{\bP}(z_1,\cdots,z_n)
=(z_1+\bp_1^T \bx,\cdots,z_n+\bp_n^T\bx)~ (mod~1),
\label{eq:defT}
\end{align}
where $``mod~ 1"$ means that each coordinate remains its fractional part.
\cref{cor:ergodic} will show that $\phi^{\bx}_{\bP}$ is ergodic in probability space $(\bbT^n, \calB, \mu)$ when $\mu$ is Lebesgue measure.

\begin{proposition}
\label{cor:ergodic}
If $\bP\in\bbM^{d\times n}$,
then the parameterized translation $\phi^{\bx}_{\bP}$ defined by \cref{eq:defT}
is ergodic with respect to Lebesgue measure.
\end{proposition}

\begin{proof}
The parameterized translation $\phi^{\bx}_{\bP}$ is measure-preserving with respect to Lebesgue measure $\mu$.
The torus $\bbT^n$ is an invariant set under the translation $\phi^{\bx}_{\bP}$ since $\phi^{\bx}_{\bm P}(\bm z)\in \bbT^n$ for each $\bz\in\bbT^n$.
Let $\chi$ be a bounded measurable function invariant under $\phi^{\bx}_{\bP}$, for example, the characteristic function of an invariant set $\bbT^n$. Then, we have
\begin{align}
	\chi_{_{\bbT^n}}(\phi^{\bx}_{\bP}(\bz))
	=\chi_{_{(\phi^{\bx}_{\bP})^{-1}({\bbT^n})}}(\bz)
	=\chi_{_{\bbT^n}}(\bz).
	\label{eq:chi-inv}
\end{align}
Without causing confusion, \Cref{eq:chi-inv} can be rewritten as 
$\chi(\phi^{\bx}_{\bP}(\bz))=\chi(\bz)$. Considering the Fourier expansion of $\chi$
\begin{align*}
	\chi(\bz)=\sum_{\bk\in\bbZ^n}\hat{\chi}_{_{\bk}} e^{i\bk^T \bz},
\end{align*}
we have
\begin{align*}
	\chi(\phi^{\bx}_{\bP}(\bz))
	=\sum_{\bk\in\bbZ^n}\hat{\chi}_{_{\bk}}
	e^{i\bk^T (z_1+\bp_1^T \bx,\cdots,z_n+\bp_n^T\bx)}
	=\sum_{\bk\in\bbZ^n}\hat{\chi}_{_{\bk}}
	e^{i\bk^T (\bp_1^T \bx,\cdots,\bp_n^T\bx)}
	e^{i\bk^T \bz}.
\end{align*}
Due to the $\phi^{\bx}_{\bP}$-invariance of $\chi$ and the uniqueness of Fourier
coefficients $\hat{\chi}_{_{\bk}}$, we can obtain
\begin{align*}
	\hat{\chi}_{_{\bk}}=\hat{\chi}_{_{\bk}}
	e^{i\bk^T (\bp_1^T \bx,\cdots,\bp_n^T\bx)},
\end{align*}
\textit{i.e.,}
\begin{align*}
	\hat{\chi}_{_{\bk}}(1- e^{i\bk^T (\bp_1^T \bx,\cdots,\bp_n^T\bx)})=0.
\end{align*}
This means that $\hat{\chi}_{_{\bk}}=0$ or
\begin{align}
	\bk^T (\bp_1^T \bx,\cdots,\bp_n^T\bx)
	=(\bP\bk)^T \bx :=m \in 2\pi\bbZ.
	\label{eq:proof-tmp}
\end{align}
Since $\bp_1,\cdots,\bp_n$ are rationally independent, then for $\bk\neq \bm{0}$
and $m\in 2\pi\bbZ$, the solution $\bx$ of \cref{eq:proof-tmp} is
countable at most. Obviously, there exists $\bx_0\in\bbR^d$ such that
$(\bP\bk)^T \bx_0\notin 2\pi \bbZ$ is true for $\bk\neq \bm{0}$, then
$\hat{\chi}_{_{\bk}}=0$. Therefore, $\chi$ is a constant outside of a set measure zero which means $\phi^{\bx}_{\bP}$ is ergodic from \cref{lem:ergodic-condition}.
\end{proof}

\textbf{The proof of \cref{thm:object}.}
\begin{proof}
From the definitions of $\hf_{\bk}$ and $\hF_{\bk}$, \Cref{eq:object0} is equivalent to
\begin{align*}
	\bbint e^{-i\blam_{\bk}^T \bx}f(\bx)\,d\bx=\frac{1}{\vert \bbT^n\vert}\int_{\bbT^n}e^{-i\bk^T\by}F(\by)\,d\by,
\end{align*}
\textit{i.e.,} we need to prove that
\begin{align}
	\bbint e^{-i (\bP\bk)^T \bx}F(\bP^T \bx)\,d\bx=\frac{1}{\vert \bbT^n\vert}\int_{\bbT^n}e^{-i\bk^T\by}F(\by)\,d\by.
	\label{eq:object-tmp}
\end{align}
Denote
$
	G(\by)=e^{-i\bk^T\by}F(\by).
$
\Cref{eq:object-tmp} can be rewritten as
\begin{align}
	\bbint G(\bP^T \bx)\,d\bx=\frac{1}{\vert \bbT^n\vert}\int_{\bbT^n}G(\by)\,d\by.
	\label{eq:object}
\end{align}
According to the parameterized translation $\phi^{\bx}_{\bP}$ defined in \cref{eq:defT}, \Cref{eq:object} is equivalent to
\begin{align}
	\bbint G(\phi^{\bx}_{\bP}({\bm 0}))\,d\bx=\frac{1}{\vert \bbT^n\vert}\int_{\bbT^n}G(\by)\,d\by.
	\label{eq:object-2}
\end{align}
Applying the ergodicity of $\phi^{\bx}_{\bP}$ proved in \cref{cor:ergodic} and 
\cref{thm:Birkhoff}, \Cref{eq:object-2} is true. The proof of \cref{thm:object} is completed.
\end{proof}

We take a one-dimensional quasiperiodic function as an example to demonstrate \cref{thm:object}, which can be embedded into a two-dimensional periodic system, as shown in \cref{fig:PMmody1y2v1}. In \cref{fig:slice}, we lift the definition area (blue line) of one-dimensional quasiperiodic function to two-dimensional periodic lattice as an irrational line by a projection matrix $\bm{P}=(1,\sqrt{3})$. Then the irrational line can be reduced to a two-dimensional unit cell by modulo arithmetic due to the two-dimensional periodicity, as shown in \cref{fig:mody1} and \cref{fig:mody1y2}. The irrational slice is infinite, these moduled lines become dense in the two-dimensional unit cell. Therefore, as \cref{thm:object} states, the one-dimensional quasiperiodic Fourier coefficient can be replaced by the two-dimensional parent Fourier coefficient.

\begin{figure}[!htbp]
\centering
\subfigure[The line $\bm{P}^T x$ embedded in a two-dimensional periodic lattice]{
	\label{fig:slice}
	\includegraphics[width=1.4in]{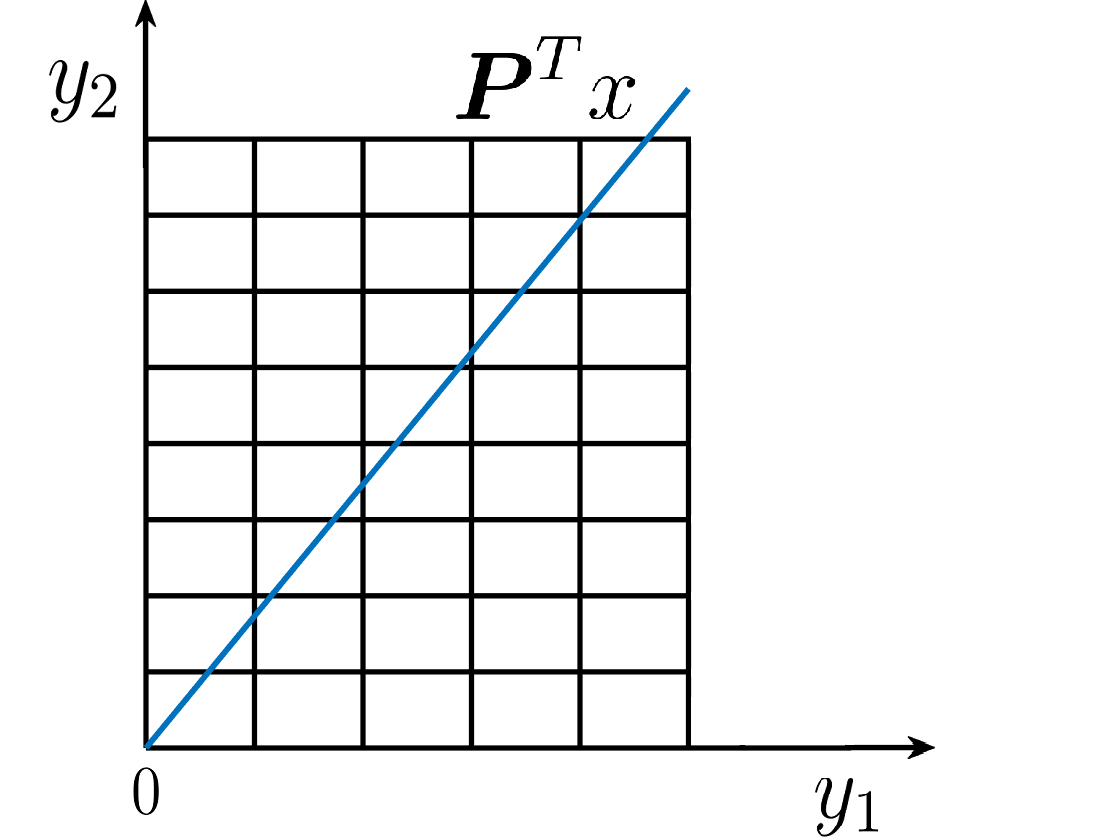}}~~~
\subfigure[Step 1: Modulo $\bm{P}^T x$ along $y_1$-axis]{
	\label{fig:mody1}
	\includegraphics[width=1.4in]{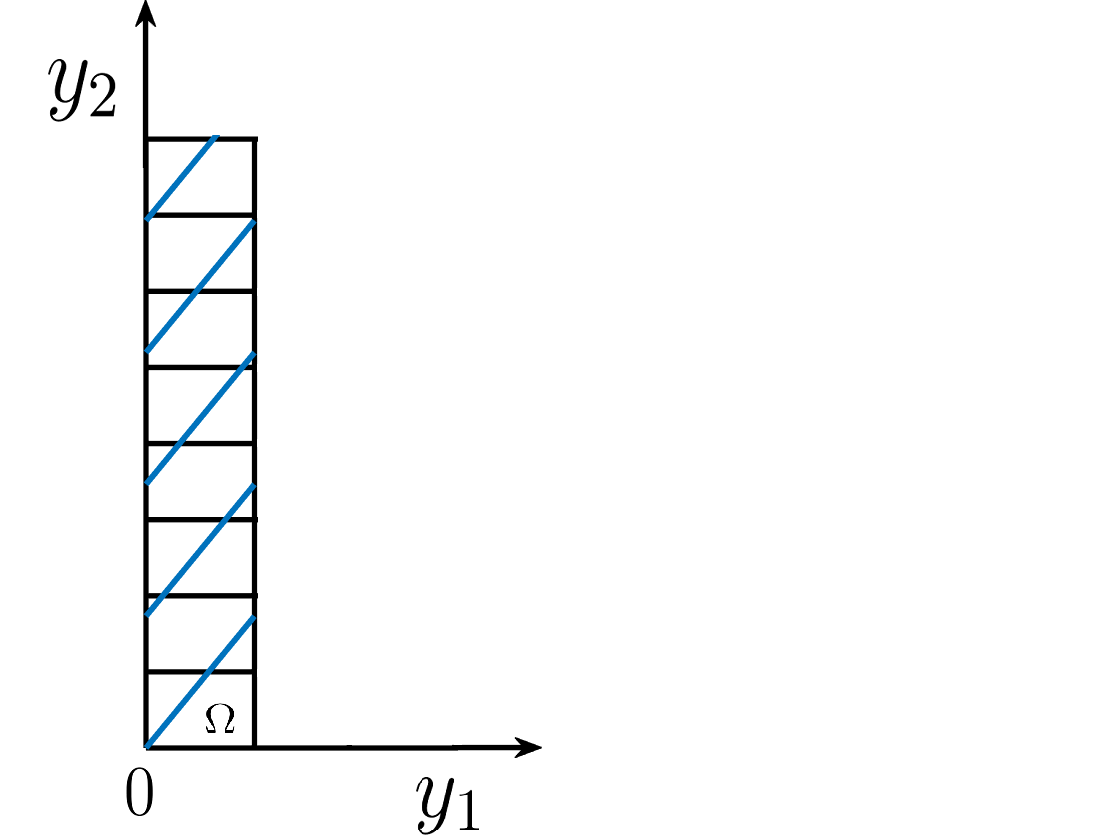}}~~
\subfigure[Step 2: Modulo $\bm{P}^T x$ along $y_2$-axis after Step 1]{
	\label{fig:mody1y2}
	\includegraphics[width=1.4in]{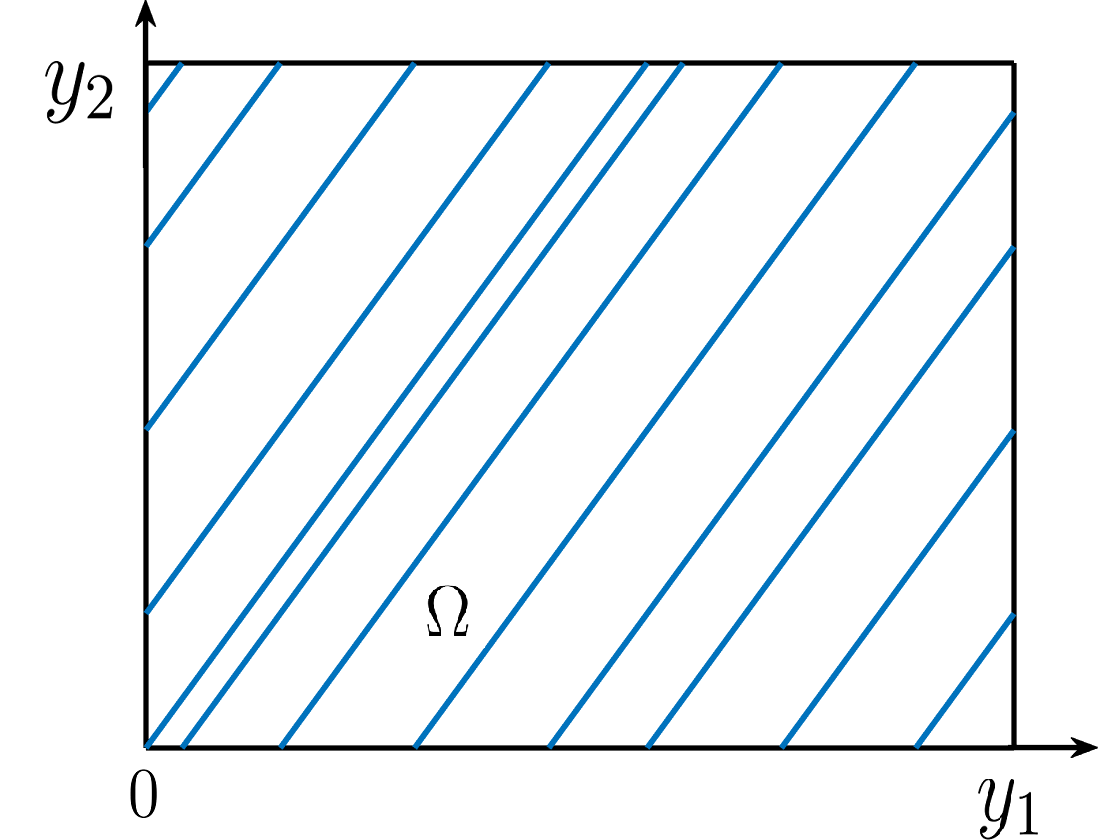}}
\caption{The process of modulo a two-dimensional irrational slice $\bm P^T x$ where $\bm P=(1,\sqrt{3})$, $x\in\mathbb{R}$.}
\label{fig:PMmody1y2v1}
\end{figure}


Applying \cref{thm:object}, we have the following two corollaries.

\begin{corollary}
	\label{cor:uniqueness}
	Quasiperiodic function $f(\bx)$ and its parent function are uniquely determined each other when the projection matrix $\bP$ is given.
\end{corollary}

\begin{proof}
	On the one hand, when the parent function and the projection matrix $\bP$ are given, the quasiperiodic function $f(\bx)$ is obviously unique.
	
	On the other hand, we prove that when the projection matrix $\bP$ is given, $f(\bx)$ has a unique parent function. Assume that there exist two distinct parent functions $F(\by)$ and $G(\by)$ such that
	\begin{align*}
		f(\bx)=F(\bP^{T}\bx),~~f(\bx)=G(\bP^{T}\bx).
	\end{align*}
	From \cref{thm:object}, we can obtain $\hF_{\bk}=\hf_{\bk}=\hat G_{\bk},~ \bk\in\bbZ^n$, where $\hF_{\bk}$ and $\hat G_{\bk}$ are obtained by the continuous Fourier-Bohr transform, respectively.
	According to the uniqueness theorem\,\cite{levitan1982almost}, then it follows that $F(\by)\equiv G(\by)$.
\end{proof}

Note that the uniqueness theorem in Bohr's work states that the quasiperiodic function is uniquely determined by quasiperiodic Fourier coefficients, which are obtained by the continuous Fourier-Bohr transform \cite{bohr1947almost}. In contrast, \cref{cor:uniqueness} states the uniqueness property that arises from the relation between the quasiperiodic function and its parent function.
	


Furthermore, we can establish an isomorphism relation between quasiperiodic function space and its parent function space.
Denote
\begin{align*}
	\Tri(\bbT^n)=\Big\{	F(\by)=\sum_{\bk\in\bbZ^n}\hc_{\bk}e^{i\bk^T \by},~\by\in\bbT^n:\sum_{\bk\in\bbZ^n} \vert \hc_{\bk}\vert<\infty\Big\}.
\end{align*}
For a given projection matrix $\bP\in\bbM^{d\times n}$, we define the subspace of $\QP(\bbR^d)$
\begin{align*}
	W_{\bP}(\bbR^d)=\{f(\bx)\in\bbC(\bbR^d): f(\bx)=F(\bP^T\bx),~F\in \Tri(\bbT^n),~\bP\in \bbM^{d\times n}\}.
\end{align*}
Define a mapping $\varphi_{\bP}: \Tri(\bbT^n)\mapsto W_{\bP}(\bbR^d)$, then we can easily prove that
$\varphi_{\bP}$ is isomorphic from \cref{cor:uniqueness}.





\begin{corollary}
	\label{pro:relation-space}
	For a given function $f(\bx)\in\QP(\bbR^d)$, $F(\by)$ is its parent function, we have the following properties
	
	(i) $F(\by)\in L^{\infty}(\bbT^n)$ if and only if $f(\bx)\in\QP_1(\bbR^d)$.
	
	(ii) $F(\by)\in L^2(\bbT^n)$ if and only if $f(\bx)\in\QP_2(\bbR^d)$.
	
\end{corollary}

\begin{proof}
	For $f(\bm x)\in \QP_1(\bbR^d)$, we have
	$
			f(\bx) =\sum_{\bk\in \bbZ^{n}}\hf_{\bk} e^{i(\bP\bk)^T \bx}.
	$
	Denote the periodic function
$
			g(\by)=
			\sum_{\bk\in \bbZ^{n}} \hf_{\bk}
			e^{i  \bk^T\by}.
$
	Obviously, $f(\bx)=g(\bP^{T}\bx)$, \textit{i.e.,} $g(\by)$ is the parent function of $f(\bx)$. Applying \cref{cor:uniqueness} and \cref{thm:object} leads to
		\begin{align}
			F(\by)=g(\by)=
			\sum_{\bk\in \bbZ^{n}} \hF_{\bk}
			e^{i  \bk^T\by},
			\label{eq:perFourierSeries}
		\end{align}
the Fourier coefficient $\hF_{\bk}$ is calculated by \cref{eq:rasiedFC} and the Fourier series of the parent function $F(\by)$ is convergent, \textit{i.e.,} $F(\by)\in L^{\infty}(\bbT^n)$.
	Similarly, we can prove that the conclusion (i) is sufficient. 
	The conclusion (ii) can be proved similarly.
\end{proof}







Applying the Parseval's equality \cref{eq:Parseval} and \cref{pro:relation-space}, for any $F_1,~F_2\in \Tri(\bbT^n)$ and $f_1,~f_2\in W_{\bP}(\bbR^d)$, we have
\begin{align*}
\Vert F\Vert_{L^2}^2=\sum_{\bk\in\bbZ^n}\vert \hc_{\bk}\vert^2, ~~\Vert f\Vert_{\mathcal{L}^2(\bbR^d)}^2=\sum_{\bk\in\bbZ^n}\vert \hc_{\bk} \vert^2.
\end{align*}
Therefore, $\Vert f\Vert_{\mathcal{L}^2(\bbR^d)}=\Vert \varphi_{\bP} F\Vert_{\mathcal{L}^2(\bbR^d)}=\Vert F\Vert_{L^2}$,
\textit{i.e.,} $\varphi_{\bP}$ is an isometric mapping in the sense of $\mathcal{L}^2(\bbR^d)$. 
The isomorphic mapping $\varphi_{\bm P}$ is a useful tool for error estimates of QSM and PM, see \cref{thm:truncation error-2} and \cref{thm:pm error-2}, respectively.


\section{Error estimate}
\label{sec:error}

\subsection{Error analysis of QSM}
\label{sec:error-QSM}


The error analysis of QSM is built on the relation between the quasiperiodic function and its parent function. Therefore, we first give the truncation error of periodic Fourier spectral method\,\cite{C2006Spectral}.

\begin{lemma}
	\label{lem:trun-periodic}
	For each $F\in H^{\alpha}(\bbT^n)$.
	There exists a constant $C$, independent of $F$ and $N$, such that
	\begin{align*}
		\Vert \calP_N F-F \Vert_{L^2}\leq C N^{\mu-\alpha}\vert F\vert_{\alpha}. 
	\end{align*}
\end{lemma}
In the following, we will state the error estimate of QSM in $\mathcal{L}^2(\bbR^d)$- and $\mathcal{L}^{\infty}(\bbR^d)$-norm sense, respectively.
\begin{thm}
	\label{thm:truncation error-2}
Suppose that $f(\bx)\in \QP(\bbR^d)$ and its parent function
	$F(\by)\in H^{\alpha}(\mathbb{T}^n)$ with  $\alpha \geq 0$. Then, there exists a
	constant $C$, independent of $F$ and $N$, such that		
	\begin{align*}
		\|\calP_N f-f\|_{\mathcal{L}^2(\bbR^d)} \leq CN^{-\alpha}\vert F \vert_\alpha.
	\end{align*}		
\end{thm}

\begin{proof}
Obviously, \cref{pro:relation-space} implies $f\in \QP_2(\bbR^d)$. Since the mapping $\varphi_{\bP}$ is isometric in the sense of $\mathcal{L}^2(\bbR^d)$, from \cref{lem:trun-periodic}, we have
	\begin{align*}
		\Vert \calP_N f-f\Vert_{\mathcal{L}^2(\bbR^d)}
		& =\Vert \calP_N \varphi_{\bP}F -\varphi_{\bP} F\Vert_{\mathcal{L}^2(\bbR^d)}
		=\Vert \varphi_{\bP}\calP_N F -\varphi_{\bP} F\Vert_{\mathcal{L}^2(\bbR^d)}
		\\
		& =\Vert \calP_N F-F \Vert_{L^2}
		\leq  C N^{-\alpha}\vert F \vert_\alpha.
	\end{align*}
	This completes the proof.
\end{proof}
An another way of proving \cref{thm:truncation error-2} is presented in \Cref{Appendix:proofQSM}.

\begin{thm}
	\label{thm:truncation error}
Suppose that $f(\bx)\in \QP_1(\bbR^d)$ and its parent function
	$F(\by)\in H^{\alpha}(\mathbb{T}^n)$ with $\alpha>q> n/2$. There exists a
	constant $C_{t}$, independent of $F$ and $N$, such that
	\begin{align*}
		\|\calP_N f-f\|_{\mathcal{L}^{\infty}(\bbR^d)} \leq C_{t} N^{q-\alpha}\vert F \vert_\alpha.
	\end{align*}
\end{thm}

\begin{proof}
	Applying \cref{thm:object} and Cauchy-Schwarz inequality, we obtain
	\begin{align*}
		& \Vert \calP_N f- f \Vert_{\mathcal{L}^{\infty}(\bbR^d)}
		=\sup_{\bx\in \bbR^n}\bigg\vert 
		\sum_{\bk\in \bbZ^n/K_N^n} \hf_{\bk}e^{i (\bP\bk)^T \bx}\bigg\vert
		\leq \sum_{\bk\in \bbZ^n/K_N^n} \vert \hf_{\bk}\vert\\
		&\leq \bigg(\sum_{\bk\in \bbZ^n/K_N^n}
		(1+\|\bk \|_2^{2})^{-q} \bigg)^{1/2}
		\bigg(\sum_{\bk\in \bbZ^n/K_N^n}
		(1+\|\bk \|_2^{2})^q |\hf_{\bk}|^2\bigg)^{1/2}\\
		&\leq CN^{q-\alpha} \bigg(\sum_{\bk\in \bbZ^n/K_N^n}
		(1+\|\bk \|_2^{2})^{-q} \bigg)^{1/2} 
		\bigg(\sum_{\bk\in \bbZ^n/K_N^n}
		\|\bk \|_2^{2\alpha-2q} \cdot (1+\|\bk \|_2^{2})^q |\hf_{\bk}|^2\bigg)^{1/2}\\
		&\leq C2^{q/2} N^{q-\alpha} \bigg(\sum_{\bk\in \bbZ^n/K_N^n}
		(1+\|\bk \|_2^{2})^{-q} \bigg)^{1/2}
		\bigg(\sum_{\bk\in \bbZ^n}
		\|\bk \|_2^{2\alpha} \cdot  |\hf_{\bk}|^2\bigg)^{1/2}\\
		&= C2^{q/2}  N^{q-\alpha} \bigg(\sum_{\bk\in \bbZ^n/K_N^n}
		(1+\|\bk \|_2^{2})^{-q} \bigg)^{1/2}
		\bigg(\sum_{\bk\in \bbZ^n}
		\|\bk \|_2^{2\alpha} \cdot  |\hF_{\bk}|^2\bigg)^{1/2}\\
		&\leq C_{t} N^{q-\alpha}\vert F \vert_\alpha.
	\end{align*}
	The last inequality holds due to $\sum_{\bk\in \bbZ^n/K_N^n}
	(1+\|\bk \|_2^{2})^{-q}<\infty $ when $q>n/2$.
\end{proof}

Besides, we can also directly give the error analysis of QSM without using the parent function, see \Cref{Appendix:proofQSM2} for details. These results also show that the QSM has an exponential convergence rate.

\subsection{Error analysis of PM}
\label{sec:error-PM}

The PM grasps the essence that the quasiperiodic function can be embedded into its parent function. Assume that the Fourier series in \cref{eq:perFourierSeries} converges to $F(\by)$ at every grid point of $\bbT^n_N$. Applying the discrete orthogonality \cref{eq:disOrth}, 
the discrete parent Fourier coefficient of \cref{eq:parentdisFourierCoeff} becomes
\begin{align}
\tF_{\bk} =\langle F, e^{i \bk^T \by_{\bj}} \rangle_N
= \Big\langle \sum_{\bell\in\bbZ^n}\hF_{\bell}e^{i\bell^T \by_{\bj}} , e^{i \bk^T \by_{\bj}} \Big\rangle_N
= \hF_{\bk}+\sum_{\bmm\in\bbZ^n_*} \hF_{\bk+2N\bmm},~\bk\in K_N^n,
\label{eq:relation-DFT-FT}
\end{align}
where $\bbZ^n_*=\bbZ^n/ \{\bf{0}\}$.
Recall that $\blam_{\bk}=\bP\bk$ and from \cref{eq:relation-DFT-FT}, we have
\begin{align*}
	\sum_{\bk\in K_N^n} \tF_{\bk}e^{i\blam_{\bk}^T\bx}=\sum_{\bk\in K_N^n} \hF_{\bk} e^{i\blam_{\bk}^T\bx}+\sum_{\bk\in K_N^n} \Big(\sum_{\bmm\in\bbZ^n_*} \hF_{\bk+2N\bmm}\Big) e^{i\blam_{\bk}^T\bx}.
\end{align*}
From \cref{thm:object} and \Cref{eq:tri-quasi}, we obtain
\begin{align*}
	\sum_{\blam_{\bk} \in \bLam_N^d} \tf_{\bk}e^{i\blam_{\bk}^T\bx}
	=\sum_{\blam_{\bk}\in \bLam_N^d} \hf_{\bk} e^{i\blam_{\bk}^T\bx}
	+\sum_{\blam_{\bk}\in \bLam_N^d} \Big(\sum_{\bmm\in\bbZ^n_*} \hf_{\bk+2N\bmm}\Big) e^{i\blam_{\bk}^T \bx}.
\end{align*}
It follows that
\begin{align*}
	I_Nf=\calP_N f+R_Nf,
\end{align*}
where 
\begin{align*}
	\calP_N f=\sum_{\blam_{\bk}\in \bLam_N^d} \hf_{\bk} e^{i\blam_{\bk}^T \bx},~~
	R_Nf=\sum_{\blam_{\bk}\in \bLam_N^d} \Big(\sum_{\bmm\in\bbZ^n_*} \hf_{\bk+2N\bmm}\Big) e^{i\blam_{\bk}^T \bx}.
\end{align*}
Similar to the periodic Fourier pseudo-spectral method, $I_N f$ and $R_Nf$ represent the interpolation and the aliasing part, respectively. As a consequence, we have
\begin{align*}
	f-I_Nf=(f-\calP_N f)-R_Nf.
\end{align*}
Thus, the approximation error of PM consists of two parts, the truncation error $f-\calP_N f$ as QSM has, and the aliasing error $R_Nf$. 
The truncation error estimates of \cref{thm:truncation error-2} and
\cref{thm:truncation error} for QSM are also valid for PM. 
The aliasing error will be analyzed in the following content.
The $\mathcal{L}^2(\bbR^d)$- and $\mathcal{L}^{\infty}(\bbR^d)$-estimates of interpolation error $f-I_Nf$ are stated as follows, respectively.
\begin{thm}
	\label{thm:pm error-2}
Suppose that $f(\bx)\in\QP(\bbR^d)$ and its parent function $F(\by)\in H^{\alpha}(\mathbb{T}^n)$ with $\alpha \geq 0$.
	There exists a constant $C$, independent of $F$ and $N$, such that
	\begin{align*}
		\|I_Nf-f\|_{\mathcal{L}^2(\bbR^d)} \leq CN^{-\alpha}\vert F \vert_\alpha.
	\end{align*}
\end{thm}

\begin{proof}
	\cref{pro:relation-space} tells us that $f\in\QP_2(\bbR^d)$.
	Since $\varphi_{\bP}(\calP_N F)=\calP_N (\varphi_{\bP} F)$ and $\varphi_{\bP}(R_NF)=R_N(\varphi_{\bP} F)$, we obtain
	\begin{align*}
		\Vert I_Nf-f\Vert_{\mathcal{L}^2(\bbR^d)}
		&\leq \Vert f-\calP_N f\Vert_{\mathcal{L}^2(\bbR^d)} +\Vert R_Nf\Vert_{\mathcal{L}^2(\bbR^d)}\\
		&=\Vert \varphi_{\bP} F-\calP_N ( \varphi_{\bP} F)\Vert_{\mathcal{L}^2(\bbR^d)} +\Vert R_N(\varphi_{\bP} F)\Vert_{\mathcal{L}^2(\bbR^d)}\\
		&=\Vert \varphi_{\bP} F- \varphi_{\bP} (\calP_N  F)\Vert_{\mathcal{L}^2(\bbR^d)} +\Vert \varphi_{\bP} (R_N F)\Vert_{\mathcal{L}^2(\bbR^d)}\\
		&=\Vert F-\calP_N F\Vert_{L^2} +\Vert R_N F\Vert_{L^2}.
	\end{align*}
	Ref.\,\cite{C2006Spectral} (see its Section 5.1.3) shows that
	\begin{align*}
		\Vert R_N F\Vert_{L^2}\leq C_1 N^{-\alpha}\vert F \vert_\alpha,
	\end{align*}
	where $C_1$ is independent of $F$ and $N$.
	Then, the proof is completed by combining \cref{lem:trun-periodic}.
\end{proof}
An another way of proving \cref{thm:pm error-2} is provided in \Cref{Appendix:proofPM}.

\begin{thm}
	\label{thm:pm error}
Suppose that $f(\bx)\in \QP_1(\bbR^d)$ and its parent function
	$F(\by)\in H^{\alpha}(\mathbb{T}^n)$ with  $\alpha>q>n/2$.
	There exists a constant $C_p$, independent of $F$ and $N$, such that 
	\begin{align*}
		\|I_Nf-f\|_{\mathcal{L}^{\infty}(\bbR^d)} \leq C_p N^{q-\alpha}\vert F \vert_\alpha.
	\end{align*}	
\end{thm}

\begin{proof}
	According to the definition of $\Vert\cdot \Vert_{\mathcal{L}^{\infty}(\bbR^d)}$, we have
	\begin{align}
		\Vert I_Nf-f\Vert_{\mathcal{L}^{\infty}(\bbR^d)}
		\leq \sum_{\blam_{\bk}\in \bLam/\bLam_N^d} \vert \hf_{\bk} \vert
		+\sum_{\blam_{\bk}\in \bLam_N^d} \Big\vert \sum_{\bmm\in\bbZ^n_*} \hf_{\bk+2N\bmm}\Big\vert.
		\label{eq:PM-thm-infty-proof}
	\end{align}
	From \cref{thm:truncation error}, it follows that
	\begin{align*}
		\sum_{\blam_{\bk}\in \bLam/\bLam_N^d} \vert \hf_{\bk} \vert
		\leq C_{t}N^{q-\alpha}\vert f\vert_\alpha,
	\end{align*}
	where $\alpha>q>n/2$.
For the second term on the right side of inequality \cref{eq:PM-thm-infty-proof}, combining with Cauchy-Schwarz inequality, we can obtain
	\begin{align*}
		&\sum_{\blam_{\bk}\in \bLam_N^d} \Big\vert \sum_{\bmm\in\bbZ^n_*} \hf_{\bk+2N\bmm}\Big\vert
		=\sum_{\bk\in K_N^n} \Big\vert \sum_{\bmm\in\bbZ^n_*} \hF_{\bk+2N\bmm}\Big\vert
		\leq \sum_{\bk\in K_N^n} \sum_{\bmm\in\bbZ^n_*} \Big\vert \hF_{\bk+2N\bmm}\Big\vert\\
		& \leq \Big[\sum_{\bk\in K_N^n}
		\sum_{\bmm\in\bbZ^n_*} (1+\Vert\bk+2N\bmm \Vert_2^2)^{-\alpha}\Big]^{\frac{1}{2}}
		\cdot \Big [ \sum_{\bk\in K_N^n} \sum_{\bmm\in\bbZ^n_*}
		(1+\Vert\bk+2N\bmm \Vert_2^2)^{\alpha} \,\vert\hF_{\bk+2N\bmm}\vert^2\Big]^{\frac{1}{2}}.
	\end{align*}
For $\bk\in K_N^n$ and $\bmm\in\bbZ^n_*$, we have $\Vert\bk+2N\bmm \Vert_2^2\geq N^2$.
Furthermore, for $\alpha>q$, it follows that
\begin{align*}
&\Big[\sum_{\bk\in K_N^n}
\sum_{\bmm\in\bbZ^n_*} (1+\Vert\bk+2N\bmm \Vert_2^2)^{-\alpha}\Big]^{\frac{1}{2}}\\
&=\Big[\sum_{\bk\in K_N^n}
\sum_{\bmm\in\bbZ^n_*} (1+\Vert\bk+2N\bmm \Vert_2^2)^{q-\alpha}\cdot (1+\Vert\bk+2N\bmm \Vert_2^2)^{-q}\Big]^{\frac{1}{2}}\\
&=(1+N^2)^{\frac{q-\alpha}{2}}\cdot \Big[\sum_{\bk\in K_N^n}
\sum_{\bmm\in\bbZ^n_*} (1+\Vert\bk+2N\bmm \Vert_2^2)^{-q}\Big]^{\frac{1}{2}}\\
&\leq 2^{\frac{q-\alpha}{2}}N^{q-\alpha}\cdot \Big[\sum_{\bk\in K_N^n}
\sum_{\bmm\in\bbZ^n_*} (1+\Vert\bk+2N\bmm \Vert_2^2)^{-q}\Big]^{\frac{1}{2}}.
\end{align*}
When $q>n/2$, the series $S:=\sum_{\bk\in K_N^n}\sum_{\bmm\in\bbZ^n_*} (1+\Vert\bk+2N\bmm \Vert_2^2)^{-q}$ converges. Therefore, 
	\begin{align*}
	&\sum_{\blam_{\bk}\in \bLam_N^d} \Big\vert \sum_{\bmm\in\bbZ^n_*} \hf_{\bk+2N\bmm}\Big\vert\\
	&\leq  2^{\frac{q-\alpha}{2}}N^{q-\alpha} S^{1/2} \cdot
	\Big [\sum_{\bk\in K_N^n} \sum_{\bmm\in\bbZ^n_*}
	(1+\Vert\bk+2N\bmm \Vert_2^2)^{\alpha} \,\vert\hF_{\bk+2N\bmm}\vert^2\Big]^{\frac{1}{2}}\\
	&\leq  2^{q/2} N^{q-\alpha} S^{1/2} \vert F \vert_\alpha
	=C_{a} N^{q-\alpha} \vert F \vert_\alpha.
\end{align*}
	Then, 
\begin{align*}
	\Vert I_Nf-f\Vert_{\mathcal{L}^{\infty}(\bbR^d)}
	\leq C_{t}N^{q -\alpha} \vert F \vert_\alpha+C_{a} N^{q-\alpha} \vert F \vert_\alpha
	=C_pN^{q-\alpha} \vert F \vert_\alpha.
\end{align*}

\end{proof}




\section{Application}
\label{sec:schrodinger}



In \Cref{sec:error}, we have provided prior estimates of PM and QSM. In this section, we further investigate the accuracy and efficiency of numerical methods for solving the quasiperiodic system. The TQSE with a spatially quasiperiodic solution is an important quasiperiodic system\,\cite{li2021numerical,Wang2020Localization,bourgain2004anderson,Wang2021Layer-Splitting}. Concretely, consider
\begin{align}
i\psi_t(x,t)=-\frac{1}{2}\psi_{xx}(x,t)+v(x)\psi(x,t),~~(x,t)\in \bbR\times [0,T],
\label{eq:SE}
\end{align}
with incommensurate potential $v(x)=\sum_{\lambda\in \Lambda_1}\hat{v}_{\lambda}e^{i\lambda x}$, where $\Lambda_1=\{1,-1,\sqrt{5},-\sqrt{5}\}$ and $\hat{v}_{\lambda}=1$.
Let the initial value
$
\psi_0(x)=\sum_{\lambda\in \Lambda_2} \hat{c}_{\lambda}e^{i\lambda x},~x\in \bbR,
$ with $\Lambda_2=\{\lambda = m+n\sqrt{5}: m,n\in \bbZ, -32\leq m,n\leq 31\}$ and $\hat{c}_{\lambda}=e^{-(\vert n\vert +\vert m\vert)}$. Therefore, the projection matrix is $\bP=(1,\sqrt{5})$.
The product term of wave function $\psi(x,t)$ and potential function $v(x)$, a convolution in the reciprocal space, allows us to examine the performance of different methods.

In the following, we empoly QSM, PM and PAM to discretize \cref{eq:SE} in space direction, and the second-order operator splitting (OS2) method in time direction.
In each interval $[0,2\pi)$, we use $2N$ discrete points, corresponding to the number of basis functions of QSM. 
Here we are concerned with the accuracy of spatial quasiperiodic solution, therefore, the final time $T$ can be arbitrary. For simplicity, we choose $T=0.001$. The time step size $\tau=1\times 10^{-7}$ ensures that the time truncation error does not affect the spatial approximation error.

\subsection{Numerical implementation}
\label{subsec:Methodimplement}

\subsubsection{QSM discretization}
\label{subsubsec:QSMimplement}

As \Cref{subsec:QSM} states, the QSM approximates the wave function $\psi(x,t)$ in a finite dimensional space
\begin{align*}
	\psi(x,t)\approx \calP_N \psi(x,t)
	= \sum_{\lambda\in \bLam_N}\hat{\psi}_{\lambda}(t)e^{i\lambda x}.
\end{align*}
The quasiperiodic Fourier coefficient $\hat{\psi}_\lambda$ is obtained by the continuous Fourier-Bohr transform \cref{eq:transform-FC}. 
$\bLam_N$ is defined by \cref{eq:set-trun-Lambda} with $d=1$ and $n=2$. 
$\#(\bLam_N)=(2N)^2:=D$.
Then the TQSE \cref{eq:SE} is discretized as
\begin{equation}
\begin{aligned}
i\sum_{\lambda\in \bLam_N}
\frac{d \hat{\psi}_{\lambda}(t)}{dt}e^{i\lambda x}
=\frac{1}{2}\sum_{\lambda\in \bLam_N}
\vert\lambda \vert^2\hat{\psi}_{\lambda}(t)e^{i\lambda x}
+\Big (\sum_{\lambda\in \Lambda_1}\hat{v}_{\lambda}e^{i\lambda x} \Big)
\Big(\sum_{\lambda\in \bLam_N}\hat{\psi}_{\lambda}(t)e^{i\lambda x} \Big).
\end{aligned}
	\label{eq:SE-fourier}
\end{equation}
Making the inner product of \cref{eq:SE-fourier} by $e^{i\beta x}$ and applying the
orthogonality \cref{eq:orthAP}, we obtain
\begin{align}
	i\frac{d \hat{\psi}_{\beta}(t)}{dt}
	=\frac{1}{2}\vert\beta \vert^2\hat{\psi}_{\beta}(t)
	+\sum_{\lambda\in \bLam_N}\hat{v}_{\beta-\lambda} \hat{\psi}_{\beta}(t),~~\beta\in \bLam_N.
	\label{eq:SE:QSM}
\end{align}


By applying OS2 method to semi-discrete equation \cref{eq:SE:QSM}, we can obtain the fully discrete scheme as given in \Cref{subappendix:QSM}.
Since the QSM cannot use FFT, the computational cost of solving \cref{eq:SE:QSM} in each time step is dominated by the convolution calculation with computational complexity of $O(D^2)$.



\subsubsection{PM discretization}
\label{subsubsec:PMimplement}

The PM is a generalized Fourier pseudo-spectral method. As a sequence, the PM can further discretize $x$ variable through the collocation points $x_j=\bm{P}\bm{y}_j$ with 
$\by_{\bj}=( j_1 \pi/N, j_2 \pi /N)\in \bbT_N^2,~0\leq j_1,j_2 < 2N$. 
We can expand the spatial function by discrete Fourier-Bohr expansion
\begin{align*}
	\psi(x_j,t) \approx
	I_N \psi(x_j,t)=\sum_{\blam\in \bLam_N} \tilde{\psi}_{\lambda}(t)e^{i\lambda x_j}
	~~j=0,1,\cdots,D-1,
\end{align*}
where $\tpsi_{\lambda}(t)= \tPsi_{\bk}(t)=\langle \Psi, e^{i \bk^T \by_{\bj}} \rangle_N$,
$\lambda=\bP \bk$, and $D=(2N)^2$ is the number of spatial nodes.

Denote that $V(\by)$ is the parent function of $v(x)$. Similarly, we can expand $v(x)$ using the discrete Fourier-Bohr transform. The TQSE \cref{eq:SE} is discretized as
\begin{equation}
	\begin{aligned}
		i\sum_{\lambda\in \bLam_N}
		\frac{d \tpsi_{\lambda}(t)}{dt}e^{i\lambda x_j}
		=\frac{1}{2}\sum_{\lambda\in \bLam_N}
		\vert\lambda \vert^2\tpsi_{\lambda}(t)e^{i\lambda x_j}
		+\Big (\sum_{\lambda\in \Lambda_1}\tilde{v}_{\lambda}e^{i\lambda x_j} \Big)
		\Big(\sum_{\lambda\in \bLam_N}\tpsi_{\lambda}(t)e^{i\lambda x_j} \Big),
		\label{eq:SE-fourier-discrete}
	\end{aligned}
\end{equation}
where $\tilde{v}_{\lambda}=\langle V, e^{i \bk^T \by_{\bj}} \rangle_N$.
Taking the discrete inner product of \cref{eq:SE-fourier-discrete} by $e^{i\beta
	x_j}$ and applying the discrete orthogonality \cref{eq:disOrth} yield
\begin{align}
	i\frac{d \tpsi_{\beta}(t)}{dt}
	=\frac{1}{2}\vert\beta \vert^2\tpsi_{\beta}(t)
	+\sum_{\lambda\in \bLam_N}\tilde{v}_{\beta-\lambda} \tpsi_{\lambda}(t),~\beta\in \bLam_N.
	\label{eq:PMSE}
\end{align}

Similarly, the OS2 method can be applied to discretize the semi-discrete equation \cref{eq:PMSE}. The corresponding fully discrete scheme can be found in \cref{subappendix:PM}. Meanwhile, we can use FFT to efficiently compute the convolution terms in \cref{eq:PMSE} based on the discrete Fourier-Bohr transform. Therefore, the computational complexity of PM in each time step is the level of $O(D\log D)$.



\subsubsection{PAM discretization}
\label{subsubsec:PAMimplement}

The PAM, using a periodic system to approximate the quasiperiodic systems, is a widely used approach to addressing quasiperiodic systems\,\cite{jiang2018numerical}. 
Here, we use a periodic Schr\"{o}dinger equation over a finite fundamental region $[0,2\pi L)$, $L\in\mathbb{N}_0$ to approximate TQSE. Then we can use the periodic Fourier pseudo-spectral method to solve the approximated periodic Schr\"{o}dinger equation. We use $D=2ML$ discrete points to discretize the one-dimensional periodic system. The computational complexity in each time step is at the level of $O((2ML)\log (2ML))$. 
\cref{subappendix:PAM} provides the implementation of the PAM of solving TQSE.


\subsection{Numerical results}

In this subsection, we present numerical results of solving TQSE \cref{eq:SE} by using PM, QSM and PAM. All algorithms are coded by MSVC++ 14.29 on Visual Studio Community 2019. The used FFT in PM and PAM is based on the software FFTW 3.3.5\,\cite{frigo2005design}. All computations are carried out on a workstation with an Intel Core 2.30GHz CPU, 16GB RAM.
The reference solution $\psi^*(x,T)$ is obtained by using PM with the time step size $\tau = 1 \times 10^{-7}$, the fine mesh size $h=\pi/128$ and the final time $T=0.001$.
In our numerical results, we mainly show the numerical error $e_N$ and CPU time of three algorithms. Firstly, we give the calculation formula of $e_N$ of QSM, PM and PAM. Denote the exact solution of TQSE
\begin{align*}
\psi^*(x,T)=\sum_{\lambda\in \Lambda} \tilde{\psi}_{\lambda}^*(T)e^{i\lambda x}.
\end{align*}	
In the QSM, from the Parseval's equality, the numerical error is 
\begin{align*}
e_N^2&=\Vert\psi^*(x,T)-\calP_N \psi(x,T)\Vert^2_{\mathcal{L}^2(\bbR)}\\
	&=\lim_{K\rightarrow +\infty}
	\frac{1}{2K}\int^K_{-K} \vert\psi^*(x,T)-\calP_N \psi(x,T) \vert^2\,dx\\
	&=\sum_{\lambda\in \Lambda_N}\vert \tilde{\psi}_{\lambda}^*(T)-\hat{\psi}_{\lambda}(T)\vert^2.
\end{align*}
In the PM, we can obtain
\begin{align*}
e_N^2=\Vert\psi^*(x,T)-I_N \psi(x,T)\Vert^2_{\mathcal{L}^2(\bbR)}
=\sum_{\lambda\in \Lambda_N}\vert \tilde{\psi}_{\lambda}^*(T)-\tilde{\psi}_{\lambda}(T)\vert^2.
\end{align*}
In the PAM, assume that the exact solution of the periodic Schr\"{o}dinger system \eqref{eq:approxSE} is
\begin{align*}
\varphi^*(x, T)=\sum_{k\in \bbZ}\tilde{\psi}^*_k(T)e^{ik x},~~x\in [0,2\pi L).
\end{align*}
The numerical solution obtained by PAM is
\begin{align*}
\varphi_M(x,T)=\sum_{k\in \Lambda_M^{PAM}} \tilde{\varphi}_k(T)e^{ik x},~~x\in [0,2\pi L),
\end{align*}
where $\Lambda^{PAM}_M=\{k\in \bbZ: -LM\leq k <LM \}$ is a finite subset of $\bbZ$ containing a subset of $\{k\in\mathbb{Z}: k=[L\lambda ],~\lambda\in \Lambda_N\}$.
Then we can compute the numerical error
\begin{align*}
e_M^2=\Vert \varphi^*(x, T)-\varphi_M(x,T) \Vert_{L^2([0,2\pi L))}
=\sum_{k\in \Lambda_M^{PAM}}\vert  \tilde{\psi}_{k}^*(T) -\tilde{\varphi}_k(T)\vert^2.
\end{align*}
Therefore, the errors of three methods are all measured by the convergence of corresponding Fourier coefficients. Note that both QSM and PM calculate the global quasiperiodic system over $\mathbb{R}$, while the PAM only computes a periodic approximation system on a fundamental period $[0,2\pi L)$.



We present the numerical results of PAM with $M=4N$. For convenience, we use $e_N$ to replace $e_M$. Through extensive experiments, we adopt $N=8$ (also see \Cref{tab:error}) in PAM to ensure enough numerical accuracy of discretizing TQSE.
\begin{figure}[!htbp]
	\subfigure[In the PAM, the relationship between the numerical error $e_N$ and $L$ with $N=8$.]{
		\label{fig:PAMerrorN32v2}
		\includegraphics[width=4.5in]{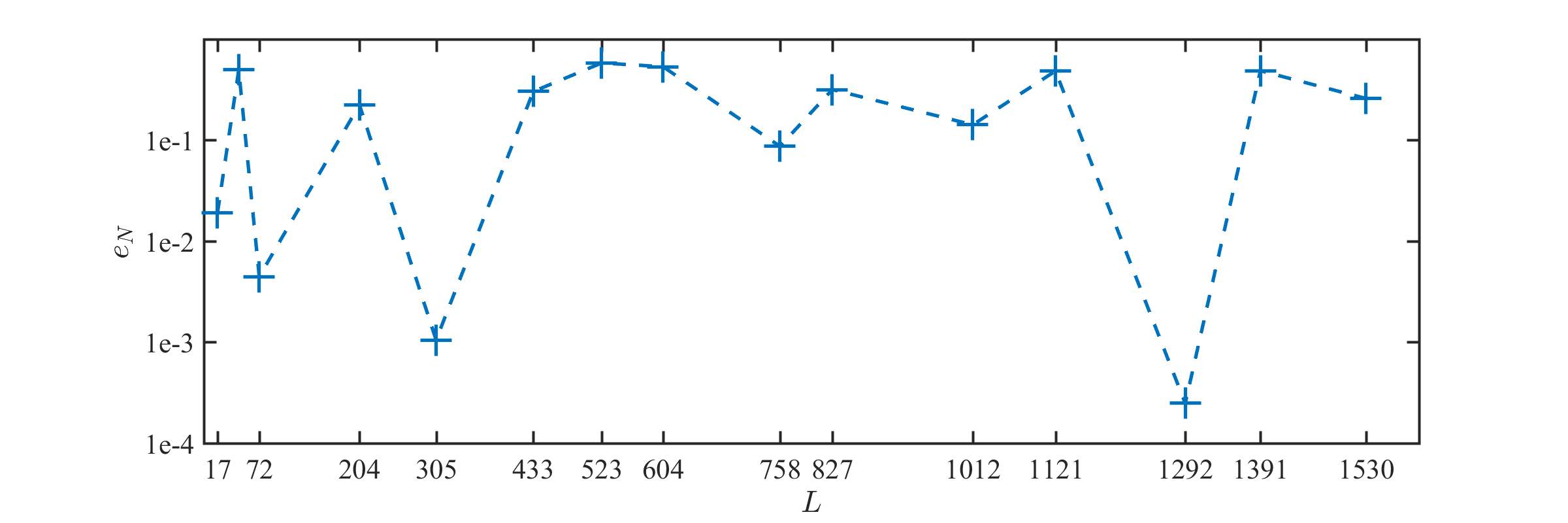}} 
	\subfigure[Diophantine approximation error]{
		\label{fig:Diophantineerror}
		\includegraphics[width=4.5in]{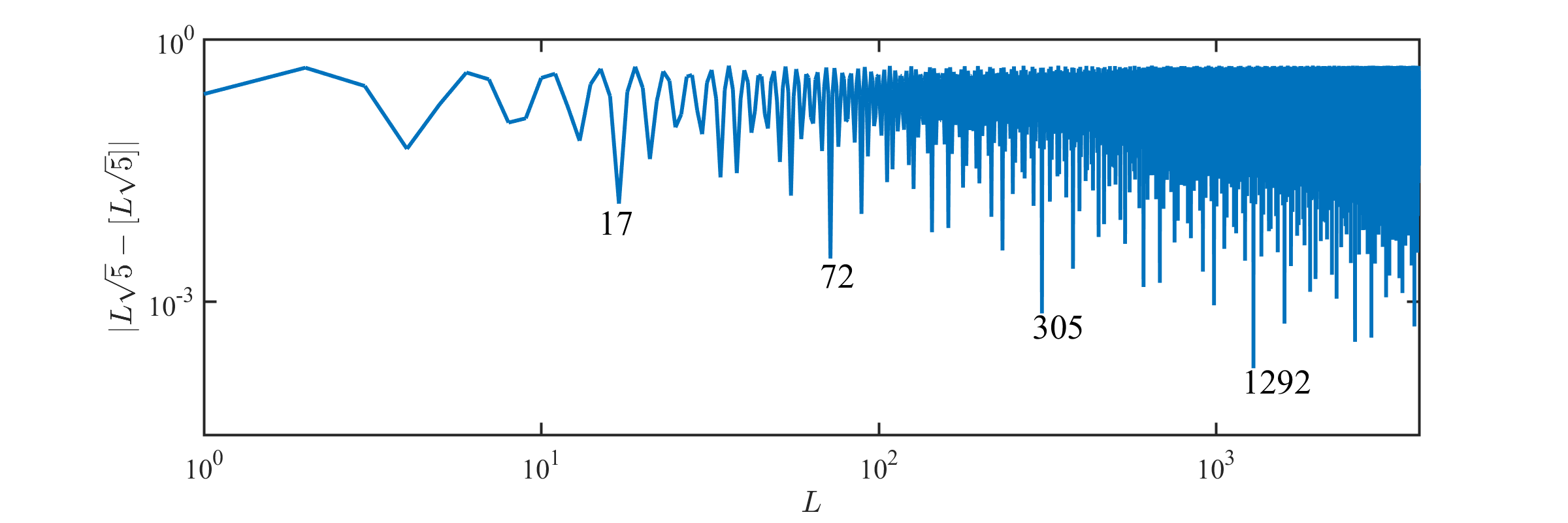}}
	\caption{Approximation error of PAM as the domain size $L$ increases.}
	\label{fig:PAM}
\end{figure}
\Cref{fig:PAMerrorN32v2} shows the approximation error obtained by PAM with $N=8$. The approximation error $e_N$ of PAM exhibits an oscillation phenomenon as domain size $L$ increases. This behavior can be attributed to the Diophantine approximation error, \textit{i.e.}, using rational numbers to approximate the irrational number. As depicted in \Cref{fig:Diophantineerror}, the Diophantine approximation error $\{L\sqrt{5}\}:=\vert L\sqrt{5}-[L\sqrt{5}] \vert $, where $[\alpha]$ denotes the nearest integer to $\alpha$, does not uniformly decrease with an increase of $L$ due to the arithmetic property of irrational number $\sqrt{5}$. Relevant function approximation theory on the PAM can refer to \cite{jiang2022pam}. For specific values of $L$, such as 17, 72, 305 and 1292, the Diophantine approximation error as well as the approximation error $e_N$ can gradually decrease. 

Then, we compare the approximation error $e_N$ of PM, QSM and PAM. \Cref{tab:error} shows $e_N$ of three algorithms as discrete points increase. \Cref{fig:error} gives a visual image to show the convergence rate.
For the PAM, we only present these results when $L=17,\,72,\,305,\, 1292$.
The approximation error of PAM consists of the quasiperiodic approximation error determined by the Diophantine approximation error $\{L\sqrt{5}\}$, and the numerical discrete error of solving periodic Schr\"{o}dinger system \eqref{eq:approxSE}.
The quasipepriodic approximation error is mainly controlled by the Diophantine approximation error.
The numerical discrete error is dependent on the discrete points. 
Once $L$ is fixed, the discrete points achieve a critical value, then $e_N$ of PAM cannot decrease, as shown in \Cref{tab:error}.
Therefore, $e_N$ of the PAM is mainly determined by the quasiperiodic approximation error. Theoretically, $e_N$ of the PAM can decrease by choosing a large and reasonable $L$. However the resulting computational cost could be unbearable. More significantly, $L$ cannot go to infinity in the numerical computation. As a result, the quasiperiodic approximation error cannot be avoided.
\Cref{tab:error} also shows that QSM and PM both have exponentially convergent rates in solving TQSE, consistent with the error estimates in \Cref{sec:error}. 
Besides, the aliasing error $\Vert R_N\psi\Vert_{\mathcal{L}^2(\bbR)}$ of PM is almost smaller than the level of $10^{-12}$, even for the $4\times 4$ grid.

\begin{table}[!hptb]
	\vspace{-0.2cm}
	\centering
	\footnotesize{
		\caption{Numerical error $e_N$ of PM, QSM and PAM for different $N$.}
		\vspace{0.001cm}
		\renewcommand\arraystretch{1.3}
		\setlength{\tabcolsep}{3mm}
		{\begin{tabular}{|c|c|c|c|c|c|}\hline
			$N$& 2   & 4 &8  &16 &32    \\ \hline
			PM  &4.132e-03  &7.569e-04 &2.543e-05 &1.702e-08 & 1.748e-12\\ \hline
			QSM &4.132e-03  &7.569e-04&2.543e-05 &1.702e-08 & 1.903e-12  \\ \hline
		    PAM ($L=17$) &1.907e-02 &1.900e-02 & 1.899e-02&1.899e-02&1.899e-02 \\ \hline
			PAM ($L=72$)&4.536e-03 &4.449e-03 & 4.449e-03& 4.449e-03& 4.449e-03 \\ \hline
			PAM ($L=305$) &1.376e-03 &1.052e-03 & 1.051e-03& 1.051e-03& 1.051e-03\\ \hline
			PAM ($L=1292$)&9.219e-04 &2.529e-04 & 2.480e-04& 2.480e-04& 2.480e-04\\ \hline
		\end{tabular}}
		\label{tab:error}
	}
\end{table}

\begin{figure}[!htbp]
	\centering
	\subfigure{
		\includegraphics[width=4.5in]{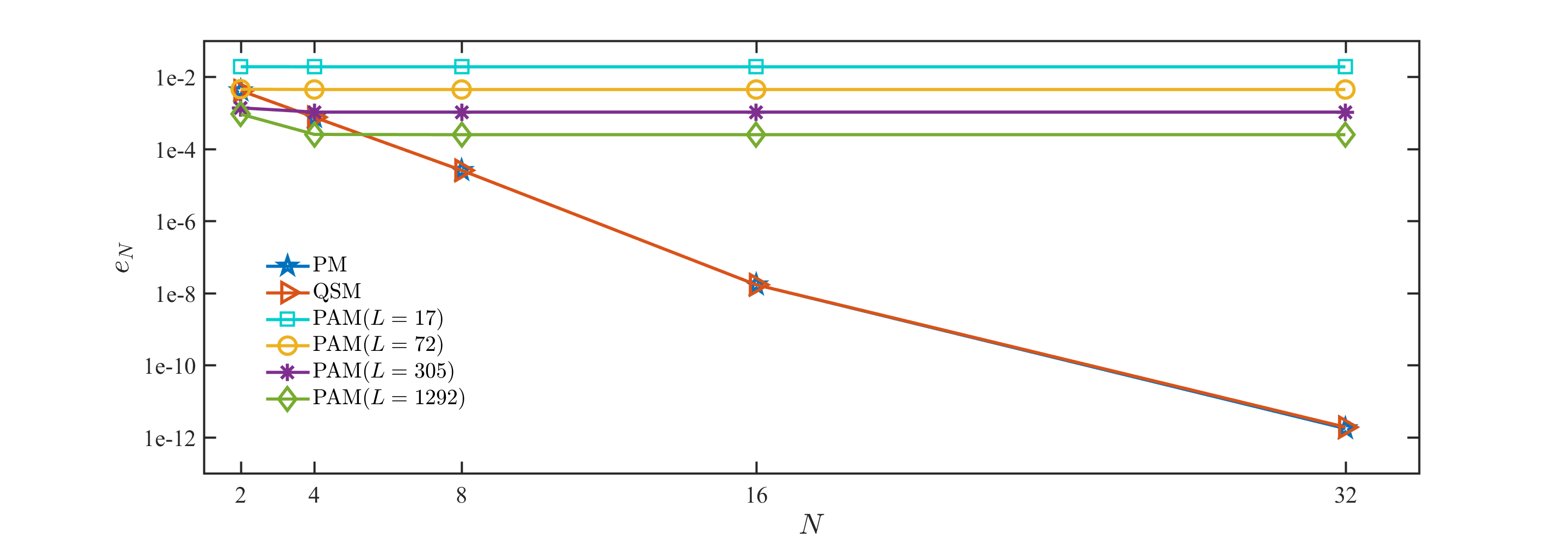}}
	\caption{The relationship between the numerical error $e_N$ and $N$.}\label{fig:error}
\end{figure}



We examine the efficiency of three methods by comparing CPU time in solving TQSE, as shown in \cref{tab:CPU}. These results demonstrate that the CPU time required by QSM increases dramatically with an increase of $N$ due to the invalidity of FFT. In contrast, the PM can greatly save computational amounts by using FFT. The CPU time of PAM has a similar behavior to PM due to the availability of FFT. However, the PAM is less efficient than PM since the PAM needs more discrete nodes.

\begin{table}[!hptb]
	\vspace{-0.2cm}
	\centering
	\footnotesize{
		\caption{Required CPU time (s) of PM, QSM and PAM for different $N$.}\label{tab:CPU}
		\vspace{0.001cm}
		\renewcommand\arraystretch{1.3}
		\setlength{\tabcolsep}{3mm}
		{\begin{tabular}{|c|c|c|c|c|c|}\hline
				$N$  & 2& 4 &8  &16 &32  \\ \hline
				PM   &0.051 & 0.077    &0.237    &0.716     & 2.873       \\ \hline
				QSM  &0.125 & 1.020  &13.366    &198.301    &3347.355 \\ \hline
	    PAM ($L=17$) &0.331 &0.593 &1.146&2.554&4.204 \\ \hline
		PAM ($L=72$) &0.994 &1.833 &3.741&7.382&15.947 \\ \hline
	   PAM ($L=305$) &6.497 &12.853 &27.451&64.089&109.709\\ \hline
	  PAM ($L=1292$) &28.625 &50.074 & 114.273&247.594&494.179 \\ \hline
		\end{tabular}}
	}
\end{table}

Finally, combining the data in Table \ref{tab:error} and Table \ref{tab:CPU}, we plot the relationship between $e_N$ and CPU time in Figure \ref{fig:errorcpu}. These results show that the PM is a high-precision and efficient algorithm in solving TQSE \cref{eq:SE}.


\begin{figure}[!htbp]
	\centering
	\subfigure{
		\includegraphics[width=4.5in]{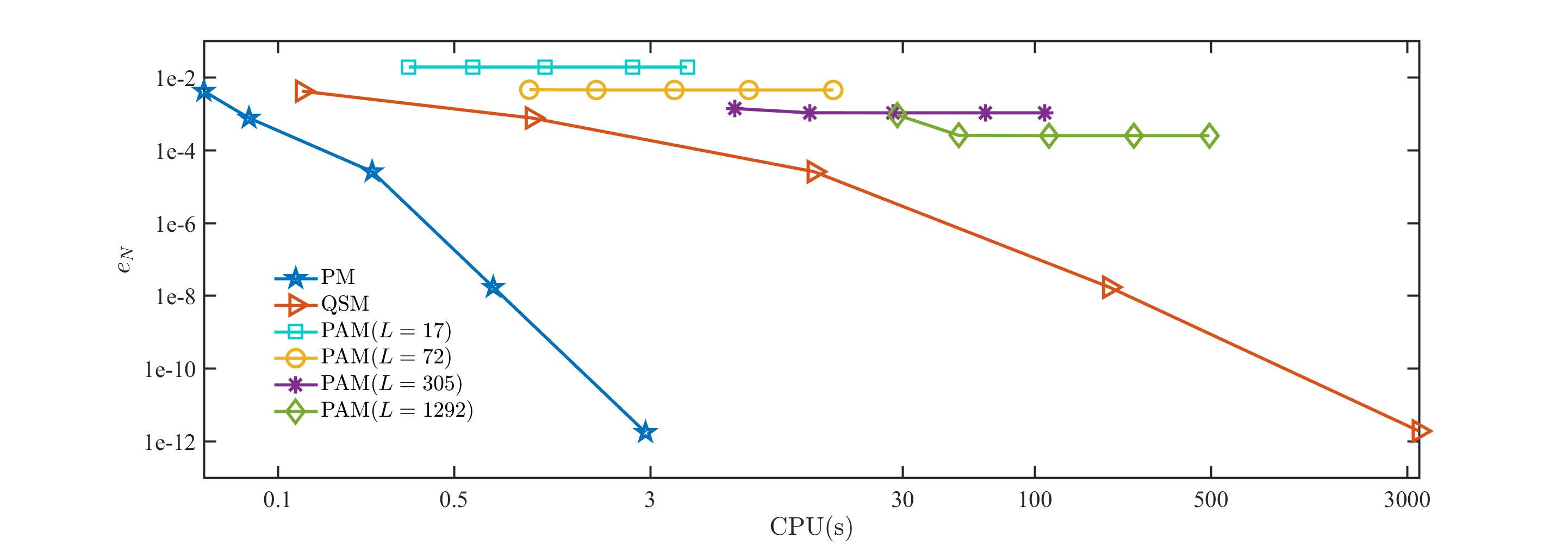}}
	\caption{The relationship between the numerical error $e_N$ and CPU time (s) when $N=2,4,8,16,32,$ respectively.}\label{fig:errorcpu}
\end{figure}


\section{Discussion and Conclusions}
\label{sec:conclusion}




In this paper, we present the convergence analysis of PM and QSM by revealing the relation between quasiperiodic functions and their parent functions. These results demonstrate that PM and QSM have exponential decay both in $\mathcal{L}^2(\bbR^d)$- and $\mathcal{L}^{\infty}(\bbR^d)$-norm, and QSM (PM) is an extension of periodic Fourier spectral (pseudo-spectral) method. We also analyze the computational complexity of these methods. The PM can use FFT while the QSM cannot. Finally, we adopt a one-dimensional TQSE to show the accuracy and efficiency of PM, QSM and PAM in solving quasiperiodic systems. Numerical results demonstrate that PM and QSM also have exponential convergence, while the approximation error of PAM is mainly dominated by Diophantine approximation error.  
These results show that the PM is an accurate and efficient method for solving quasiperiodic systems. It is the first theoretical work of the PM. 
This work encourages us to further investigate the error estimates of PM and QSM in general function space, as well as the development of advanced numerical methods and theories for solving more quasiperiodic systems.



\appendix
\begin{appendices}
	
\section{The proof of \cref{thm:truncation error-2}}
\label{Appendix:proofQSM}

\begin{proof}
For $\bk\in K^n_N$, it follows that $\Vert \bk\Vert_2\leq \sqrt{n}N$.
By Cauchy-Schwarz inequality and applying \cref{thm:object}, we have
	\begin{align*}
		\Vert\calP_N f - f\Vert^2_{\mathcal{L}^2(\bbR^d)}
		&= \sum_{\bk\in \bbZ^n/K^n_N} \vert \hf_{\bk}\vert^2
		\leq CN^{-2\alpha} \sum_{\bk\in \bbZ^n/K^n_N}
		\|\bk \|_2^{2\alpha} |\hf_{\bk}|^2\\
		&=CN^{-2\alpha} \sum_{\bk\in \bbZ^n/K^n_N}
		\|\bk \|_2^{2\alpha} |\hF_{\bk}|^2
		\leq CN^{-2\alpha}\vert F \vert^2_\alpha.
	\end{align*}
	This completes the proof.
\end{proof}

\section{Error analysis of QSM without the help of parent functions}
\label{Appendix:proofQSM2}
Here we present an approximation analysis of QSM in the quasiperiodic function space by imposing some assumptions on the projection matrix.
\begin{theorem}
Suppose that $f(\bx)\in H^{\alpha}_{QP}(\bbR^d)$ and the nonzero minimum singular value $\sigma_{\min}(\bm P)$ of the projection matrix $\bm P$ satisfies $\sigma_{\min}(\bm P)>\theta >0$. Then, there exists a constant $C(\theta)$, independent of $f$ and $N$, such that		
	\begin{align*}
		\|\mathcal{P}_N f-f\|_{\mathcal{L}^2(\bbR^d)} \leq C(\theta) N^{-\alpha}\vert f \vert_\alpha.
	\end{align*}
\end{theorem}

\begin{proof}
	For $\bk\in K^n_N$, it follows that $\Vert \bk\Vert_2\leq \sqrt{n}N$.
	By Cauchy-Schwarz inequality, we have
	\begin{align*}
		&\Vert \mathcal{P}_N f-f\Vert^2_{\mathcal{L}^2(\bbR^d)}
		= \sum_{\bk\in \bbZ^n/K^n_N} \vert \hat f_{\bk}\vert^2
		= \sum_{\bm\lambda\in \bm\Lambda /\bm\Lambda^d_N} \vert \hat f_{\bm\lambda}\vert^2\\
		&\leq  C(\sigma_{\min}(\bm P) N)^{-2\alpha} \sum_{\bm\lambda\in \bm\Lambda /\bm\Lambda^d_N}
		\|\bm\lambda \|_2^{2\alpha} |\hat f_{\bm\lambda}|^2
		\leq C(\theta)N^{-2\alpha}\vert f \vert^2_\alpha.
	\end{align*}
	This completes the proof.
\end{proof}

\begin{theorem}
	Suppose that $f(\bx)\in H^{\alpha}_{QP}(\bbR^d)$, that the nonzero minimum singular value $\sigma_{\min}(\bm P)$ of the projection matrix $\bm P$ satisfies $\sigma_{\min}(\bm P)>\theta >0$ and $\alpha>q>d/2$.
	Then, there exists a
	constant $C(\theta)$, independent of $f$ and $N$, such that
	\begin{align*}
		\|\mathcal{P}_N f-f\|_{\mathcal{L}^{\infty}(\bbR^d)} \leq C(\theta) N^{q-\alpha}\vert f \vert_\alpha.
	\end{align*}
\end{theorem}

\begin{proof}
	Applying Cauchy-Schwarz inequality, we obtain
	\begin{align*}
		& \Vert \mathcal{P}_N f- f \Vert_{\mathcal{L}^{\infty}(\bbR^d)}
		=\sup_{\bx\in \bbR^n}\bigg\vert 
		\sum_{\bk\in \bbZ^n/K_N^n} \hat f_{\bk}e^{i (\bm P\bk)^T \bx}\bigg\vert
		\leq \sum_{\bk\in \bbZ^n/K_N^n} \vert \hat f_{\bk}\vert
		= \sum_{\bm\lambda\in \bm\Lambda /\bm\Lambda^d_N} \vert \hat f_{\bm\lambda}\vert\\
		&\leq \bigg(\sum_{\bm\lambda\in \bm\Lambda /\bm\Lambda^d_N}
		\|\bm\lambda \|_2^{-2q} \bigg)^{1/2}
		\bigg(\sum_{\bm\lambda\in \bm\Lambda /\bm\Lambda^d_N}
		\|\bm\lambda \|_2^{2q} |\hat f_{\bm\lambda}|^2\bigg)^{1/2}\\
		&= \bigg(\sum_{\bm\lambda\in \bm\Lambda /\bm\Lambda^d_N}
		\|\bm\lambda \|_2^{-2q} \bigg)^{1/2}
		\bigg(\sum_{\bm\lambda\in \bm\Lambda /\bm\Lambda^d_N}
		\|\bm\lambda \|_2^{2q-2\alpha}
		\|\bm\lambda \|_2^{2\alpha} |\hat f_{\bm\lambda}|^2\bigg)^{1/2}\\
		&\leq C[\sigma_{min}(\bm P) N]^{q-\alpha} \bigg(\sum_{\bm\lambda\in \bm\Lambda /\bm\Lambda^d_N}
		\|\bm\lambda \|_2^{-2q} \bigg)^{1/2}
		\bigg(\sum_{\bm\lambda\in \bm\Lambda /\bm\Lambda^d_N}
		\|\bm\lambda \|_2^{2\alpha} |\hat f_{\bm\lambda}|^2\bigg)^{1/2}\\
		&\leq C[\sigma_{min}(\bm P) N]^{q-\alpha} \bigg(\sum_{\bm\lambda\in \bm\Lambda /\bm\Lambda^d_N}
		\|\bm\lambda \|_2^{-2q} \bigg)^{1/2}
		\bigg(\sum_{\bm\lambda\in \bm\Lambda}
		\vert \bm\lambda \vert^{2\alpha} |\hat f_{\bm\lambda}|^2\bigg)^{1/2}\\
		&= C(\theta) N^{q-\alpha}\vert f \vert_\alpha.
	\end{align*}
	The last inequality holds due to $\sum_{\bm\lambda\in \bm\Lambda /\bm\Lambda^d_N}
	\|\bm\lambda \|_2^{-2q}<\infty$ when $q>d/2$.
\end{proof}

\section{Another proof of \cref{thm:pm error-2}}
\label{Appendix:proofPM}
According to the definition of $\mathcal{L}^2(\bbR^d)$-norm, we have
\begin{align*}
	\Vert f-I_Nf\Vert^2_{\mathcal{L}^2(\bbR^d)}
	&=\sum_{\blam_{\bk} \in \bLam/\bLam_N^d} \vert \hf_{\bk} \vert^2
	+\sum_{\blam_{\bk}\in \bLam_N^d} \Big\vert \sum_{\bmm\in\bbZ^n_*} \hf_{\bk+2N\bmm}\Big\vert^2 \\
	&=\Vert f-\calP_N f\Vert^2_{\mathcal{L}^2(\bbR^d)}+\Vert R_Nf\Vert^2_{\mathcal{L}^2(\bbR^d)}.
\end{align*}

Recall that $\Vert \bk\Vert_2^2=\sum^n_{j=1}\vert k_j\vert^2$
and by Cauchy-Schwarz inequality, we have
\begin{align*}
	\Big\vert \sum_{\bmm\in\bbZ^n_*} \hf_{\bk+2N\bP\bmm}\Big\vert^2
	&=\Big\vert \sum_{\bmm\in\bbZ^n_*} \hF_{\bk+2N\bmm}\Big\vert^2\\
	&=\Big\vert \sum_{\bmm\in\bbZ^n_*} (1+\Vert\bk+2N\bmm \Vert_2^2)^{-\frac{\alpha}{2}}\cdot
	(1+\Vert\bk+2N\bmm \Vert_2^2)^{\frac{\alpha}{2}}\, \hF_{\bk+2N\bmm}\Big\vert^2\\
	&\leq \sum_{\bmm\in\bbZ^n_*} (1+\Vert\bk+2N\bmm \Vert_2^2)^{-\alpha}\cdot
	\sum_{\bmm\in\bbZ^n_*}
	(1+\Vert\bk+2N\bmm \Vert_2^2)^{\alpha} \,\vert\hF_{\bk+2N\bmm}\vert^2.
\end{align*}
Since $\vert k_j\vert\leq N,~j=1,\cdots, n$, for $\vert m_j \vert\geq 1$, it follows that
\begin{align*}
	\vert k_j+2N m_j\vert\geq 2N\vert m_j\vert-\vert k_j\vert \geq (2\vert m_j \vert-1)N>1.
\end{align*}
Thus, for $\bmm \in\bbZ^n$ with $\vert m_j \vert\geq 1$, we have 
\begin{align*}
	&(1+\Vert\bk+2N\bmm \Vert_2^2)^{-\alpha}
	=\Big[1+\sum^n_{j=1}\vert k_j+2N m_j\vert^2\Big]^{-\alpha}\\
	&\leq \Big[1+\sum^n_{j=1}((2\vert m_j \vert-1)N)^2\Big]^{-\alpha}
	\leq N^{-2\alpha}\Big[\sum^n_{j=1}(2\vert m_j \vert-1)^2\Big]^{-\alpha}.
\end{align*}
Then
\begin{align*}
	\sum_{\bmm\in\bbZ^n_*}(1+\Vert\bk+2N\bmm \Vert_2^2)^{-\alpha}
	&\leq N^{-2\alpha}\sum^{n}_{r=1}2^rC^r_n\sum^{+\infty}_{m_1=1} \cdots \sum^{+\infty}_{m_r=1}
	\Big[\sum^r_{j=1}(2\vert m_j \vert-1)^2\Big]^{-\alpha}.
\end{align*}
When $\alpha>r/2$, the series $\sum^{+\infty}_{m_1=1} \cdots \sum^{+\infty}_{m_r=1}
\Big[\sum^r_{j=1}(2\vert m_j \vert-1)^2\Big]^{-\alpha}$ converges.
For $\alpha>n/2$, we have
\begin{align*}
	S: =\sum_{\bmm\in\bbZ^n_*}\Big[\sum^d_{j=1}(2\vert m_j \vert-1)\Big]^{-\alpha}<\infty.
\end{align*}
Therefore, 
\begin{align*}
	\Vert R_Nf\Vert^2_{\mathcal{L}^2(\bbR^d)}
	&=\sum_{\bk\in K^n_N} \Big\vert \sum_{\bmm\in\bbZ^n_*} \hF_{\bk+2N\bmm}\Big\vert^2\\
	&\leq  N^{-2\alpha} S\cdot
	\sum_{\bk\in K^n_N}
	\sum_{\bmm\in\bbZ^n_*}
	(1+\Vert\bk+2N\bmm \Vert_2^2)^{\alpha} \,\vert\hF_{\bk+2N\bmm}\vert^2\\
	&\leq N^{-2\alpha} S \cdot 2^\alpha
	\sum_{\bk\in K^n_N} \sum_{\bmm\in\bbZ^n_*}
	\Vert\bk+2N\bmm \Vert_2^{2\alpha} \,\vert\hF_{\bk+2N\bmm}\vert^2\\
	&\leq 2^{\alpha} N^{-2\alpha} S\vert F\vert^2_\alpha.
\end{align*}
Applying \cref{lem:trun-periodic}, yields
\begin{align*}
	\Vert f-I_Nf\Vert_{\mathcal{L}^2(\bbR^d)}\leq CN^{-\alpha}\vert F\vert_\alpha.
\end{align*}

\section{Fully discrete scheme of TQSE \cref{eq:SE}}
\label{appendix:fullydiscre}
We apply the OS2 method to solving semi-discrete equations \cref{eq:SE:QSM} and \cref{eq:PMSE} in time direction. Meanwhile, we present the implementation details of PAM to solve TQSE \cref{eq:SE}. 
Let $\tau$ be the time size and the $m$-th time iteration step $t_m=m\tau$.

\subsection{Fully discrete scheme using the QSM}
\label{subappendix:QSM}

From $t_m$ to $t_{m+1}$, the OS2 scheme consists of three steps to solving \cref{eq:SE:QSM}.
 
Step 1: Consider the following ordinary differential equation for $t\in [t_m, t_m+ \tau/2]$, 
\begin{align}
	i\frac{d \hat{\psi}_{\beta}(t)}{dt}
	=\frac{1}{2}\vert\beta \vert^2\hat{\psi}_{\beta}(t),
	\label{subeq:operator1} 
\end{align}
with initial value $\hat{\psi}_{\beta}(t_m)$.
We can analytically solve \cref{subeq:operator1} and obtain
\begin{align}
	\hat{\phi}_{\beta}(t_{m})
	=\hat{\psi}_{\beta}(t_{m}+\frac{\tau}{2})=e^{-(i\beta^2 \tau)/4}\hat{\psi}_{\beta}(t_m).
	\label{eq:operator1-solution}
\end{align}

Step 2: Consider \Cref{subeq:operator2-F} for $t\in [t_m, t_{m+1}]$, 
\begin{align}
	i\frac{d \hat{\psi}_{\beta}(t)}{dt}
	=\sum_{\lambda\in \bLam_N}\hat{v}_{\beta-\lambda} \hat{\psi}_{\lambda}(t) := g(t, \hat{\psi}_\beta(t)),
	\label{subeq:operator2-F}
\end{align}
with initial value $\hat{\phi}_{\beta}(t_{m})$.
To address the convolution term, we apply the 
fourth-order Runge Kutta (RK4) method to solve \cref{subeq:operator2-F} in the reciprocal space.
Concretely, let $k_1= g(t_m, \hat\phi_\beta(t_m))$, $k_2 =  g(t_m+\tau/2, \hat{\phi}_{\beta}(t_m)+ \tau k_1/2)$, $k_3 = g(t_m+\tau/2, \hat{\phi}_{\beta}(t_m)+ \tau k_2/2)$, $k_4=g(t_m+\tau, \hat{\phi}_{\beta}(t_m)+ \tau k_3)$, then  
$\hat{\phi}_{\beta}(t_{m+1})=\hat{\phi}_{\beta}(t_{m})+\tau (k_1+2k_2+2k_3+k_4)/6$.

Step 3: Still consider \Cref{subeq:operator1} but with initial value $\hat{\phi}_{\beta}(t_{m+1})$ for $t\in [t_m+\tau/2, t_{m+1}]$,  then we can obtain $\hat{\psi}_{\beta}(t_{m+1})$ analytically.

Here we analyze the computational complexity for each time step. In Steps 1 and 3, the QSM can analytically solve \cref{subeq:operator1}, resulting in $D$ multiplication operators, respectively. In Step 2, due to the RK4 scheme and convolution summations, there are $8D^2+14D$ operators. 
Therefore, the computational complexity of QSM in solving \cref{eq:SE} is the level of $O(D^2)$.

\subsection{Fully discrete scheme using the PM}
\label{subappendix:PM}

Also, from $t_m$ to $t_{m+1}$, the OS2 scheme contains three steps in solving \Cref{eq:PMSE}. The Step 1 and Step 3 are similar to \Cref{subappendix:QSM}. In Step 2, we can calculate the convolution terms of \cref{eq:PMSE} by using two-dimensional FFT, we obtain
\begin{align*}
	\Phi(\by_{\bj},t_m)
	=\sum_{\bk \in K^2_N} 
	\tPhi_{\bk}(t_{m})e^{i\bk^T \by_{\bj}},
\end{align*}
where $\tPhi_{\bk}(t_{m})$ is obtained by Step 1.
Consider equation for $t\in [t_m, t_{m+1}]$
\begin{align}
	i\Psi_t=V(\by_{\bj})\Psi(\by_{\bj},t):=w(t,\Psi(\by_{\bj},t)),
	\label{eq:PM-step2}
\end{align}
where the initial value is $\Phi(\by_{\bj},t_m)$, $V(\by)$ is the parent function corresponding to $v(x)$. To make a fair comparison with QSM, we still use RK4 to solve \cref{eq:PM-step2} in physical space. Let $k_1= w(t_m,\Phi(\by_{\bj},t_m))$, $k_2 =  w(t_m+\tau/2, \Phi(\by_{\bj},t_m)+ \tau k_1/2)$, $k_3 = w(t_m+\tau/2, \Phi(\by_{\bj},t_m)+ \tau k_2/2)$, $k_4=w(t_m+\tau, \Phi(\by_{\bj},t_m)+ \tau k_3)$, then  
$\Phi(\by_{\bj},t_{m+1})=\Phi(\by_{\bj},t_{m})+\tau (k_1+2k_2+2k_3+k_4)/6$.
Again using FFT, we obtain $\tphi_{\beta}(t_{m+1})= \langle \Phi, e^{i \bk^T \by_{\bj}} \rangle_N$.


Next, we analyze the computational complexity of each time step. Similarly, the differential systems in Steps 1 and 3 can be analytically solved in the reciprocal space, resulting in $D$ multiplication operators, respectively. In Step 2, due to the availability of FFT, the convolutions in \cref{eq:PMSE} can be economically calculated in physical space as dot product as shown in \cref{eq:PM-step2} which rises $O(D\log D)$ operators. 
Therefore, the computational complexity of PM in solving \cref{eq:SE} is the level of $O(D\log D)$.



\subsection{Implementation of PAM of solving TQSE \cref{eq:SE}}
\label{subappendix:PAM}

We give the implementation of PAM to solve TQSE \cref{eq:SE}. 
In the PAM, we use a one-dimensional periodic Schr\"{o}dinger equation (PSE) to approximate TQSE \cref{eq:SE} over a finite region $[0,2\pi L)$. Concretely, we use the periodic functions $u(x)$ and $\varphi(x,t)$ to approximate $v(x)$ and $\psi(x,t)$, respectively. Denote
\begin{align*}
\Lambda(u)=\{h\in\bbZ: h=[ L\lambda ],~ \lambda\in \Lambda_1\},
\end{align*}
then
\begin{align*}
u(x)=\sum_{h\in \Lambda(u)}\hat{u}_{h}e^{i h x},~x\in [0,2\pi L),
\end{align*}
where $\hat{u}_{h}=\hat{u}_{[ L\lambda ]}=\hat{v}_{\lambda}=1$.
Therefore, the PAM solves the one-dimensional PSE
\begin{align}
	i\frac{d \varphi(x,t)}{dt}=-\frac{1}{2}\frac{\partial^2 \varphi(x,t)}{\partial^2 x}+u(x)\varphi(x,t),~~(x,t)\in [0,2\pi L)\times [0,T],
\label{eq:approxSE}
\end{align}
where the initial periodic function $\varphi_0(x)$ is the approximate periodic function of $\psi_0(x)$. We use the periodic Fourier pseudo-spectral method and the OS2 method to discretize \cref{eq:approxSE} in space and time directions, respectively. 
Since the PAM can use one-dimensional FFT to solve the equation \eqref{eq:approxSE} and the number of grid points is $2ML$, then the computational complexity is $O((2ML)\log (2ML))$ of each time step.

\end{appendices}

\bibliography{mybibfile}

\begin{thebibliography}{99}

\bibitem{poincare1889problem}
H. Poincar\'e, Sur le probl{\`e}me des trois corps et les {\'e}quations de la dynamique, Acta Math., 13, A3-A270, 1890.
	
	
\bibitem{shechtman1984metallic}
	D. Shechtman, I. Blech, D. Gratias and J. Cahn,
	Metallic phase with long-range orientational order and no
	translational symmetry. Phys. Rev. Lett., 53, 1951-1953, 1984.
	
\bibitem{cao2018unconventional}
Y. Cao, V. Fatemi, S. Fang, K. Watanabe, T. Taniguchi, E. Kaxiras, and P. Jarillo-Herrero, Unconventional superconductivity in magic-angle graphene superlattices. Nature, 556(7699), 43-50, 2018.
	
\bibitem{sutton1992irrational}
A. Sutton, Irrational interfaces. Prog. Mater. Sci. in materials science, 36, 167-202, 1992.
	
	\bibitem{baake2013aperiodic}
	M. Baake, and U. Grimm, Aperiodic order, Volume 1: A Mathematical Invitation. Cambridge, 2013.
	
	\bibitem{Lubensky1988aperiodicity}	
	T. Lubensky, Aperiodicity and order: introduction to quasicrystals, Vol. 1, ch. 6, edited by MV Jaric., 199-280, 1988.
	
	
\bibitem{bohr1947almost}
	H. Bohr, Almost periodic functions. Chelsea, New York, 1947.
	
\bibitem{Besicovitch1954almost}
	A. Besicovitch, Almost periodic functions. Dover Publications, New York, NY, USA, 1954.
	
\bibitem{levitan1982almost}
	B. Levitan and V. Zhikov,  Almost periodic functions and differential equations. Cambridge University Press, 1982.
	
\bibitem{jiang2018numerical}
K. Jiang and P. Zhang, Numerical mathematics of quasicrystals. Proc. Int. Cong. of Math., 3, 3575-3594, 2018.
	

\bibitem{jiang2022pam}
K. Jiang, S. Li and P. Zhang,  On the approximation of quasiperiodic functions with Diophantine frequencies by periodic functions. arXiv:2304.04334.


\bibitem{jiang2014numerical}
K. Jiang and P. Zhang, Numerical methods for quasicrystals. J. Comput. Phys., 256, 428-440, 2014.


\bibitem{barkan2014controlled}
K. Barkan, M. Engel and R. Lifshitz,
Controlled self-assembly of periodic and aperiodic cluster crystals. Phys. Rev. Lett., 113, 098304, 2014.

\bibitem{jiang2015stability}
K. Jiang, J. Tong, P. Zhang and A. Shi, Stability of
two-dimensional soft quasicrystals in systems with two length scales. Phys. Rev. E, 92, 042159, 2015.

\bibitem{zhou2019plane}
Y. Zhou, H. Chen and A. Zhou,
Plane wave methods for quantum eigenvalue problems of incommensurate systems.
J. Comput. Phys., 384, 99-113, 2019.


\bibitem{li2021numerical}
X. Li and K. Jiang, 
Numerical simulation for quasiperiodic quantum dynamical systems (in Chinese). Journal on Numerical Methods and Computer Applications, 42(1), 3-17, 2021.


\bibitem{gao2023pythagoras}
Z. Gao, Z. Xu, Z. Yang and F. Ye. Pythagoras superposition principle for localized eigenstates of two-dimensional moir\'{e} lattices. Phys. Rev. A, 108(8): 013513, 2023.


\bibitem{wang2022effective}
C. Wang, F. Liu and H. Huang, Effective model for fractional topological corner modes in quasicrystals. Phys. Rev. Lett., 129, 056403, 2022.



\bibitem{cao2021computing}
D. Cao, J. Shen and J. Xu, Computing interface with quasiperiodicity. J. Comput. Phys., 424, 109863, 2021.

\bibitem{jiang2022tilt}
K. Jiang, W. Si and J. Xu, 
Tilt grain boundaries of hexagonal structures: a spectral viewpoint. 
SIAM J. Appl. Math., 82, 1267-1286, 2022.


\bibitem{C2006Spectral}
C. Canuto, M. Hussaini, A. Quarteroni, and T. Zang, Spectral methods: fundamentals in single domains. Springer-Verlag, Berlin Heidelberg, 2006.
	
\bibitem{Bohl1893Uber}
P. Bohl, \"{U}ber die Darstellung von Funktionen einer Variablen durch trigonometrishe Reihen mit mehrer Variablen proportionalen Argumenten. Dorpat/Tartu, 1893.
	
\bibitem{Bohr1925Zur}
H. Bohr, Zur Theorie fastperiodischer Funktionen I-II. Acta Math., 45, 29-127 and 46, 101-214, 1925.
		
\bibitem{I.B.D.S.1998}
R. Iannacci, A. Bersani, G. DeIl'Acqua and P. Santucci,
Embedding theorems for Sobolev-Besicovitch spaces of almost periodic functions.
Z. Anal. Anwendungen 17(2), 443-457,1998.
	
\bibitem{Corduneanu1988almost}
C. Corduneanu, Almost periodic function. Second Edition. Chelsea, New York, 1989.
		
\bibitem{walter1982ergodic}
P. Walters, An introduction to ergodic theory. New York-Berlin: Springer-Verlag, 1982.
			
\bibitem{pitt1942generalizations}
H. Pitt, Some generalizations of the ergodic theorem. 
Math. Proc. Cambridge, 38, 325-343, 1942.
	
\bibitem{meyer1972}
Y. Meyer, Algebraic numbers and harmonic analysis. North-Holland, Amsterdam, 1972.
	
	
%
%
%
	
	
	
		
	
	\bibitem{Y.2004}
	B. Yann, Approximation by algebraic numbers. Cambridge University Press, 2004.

\bibitem{Wang2020Localization}
P. Wang, Y. Zheng, X. Chen, C. Huang, Y. Kartashov, L. Torner, V. Konotop and F. Ye, Localization and delocalization of light in photonic moir\'e lattices. Nature, 577, 42-46, 2020.

\bibitem{bourgain2004anderson}	
	J. Bourgain and W. Wang, Anderson localization for time quasi-periodic random Schr\"odinger and wave equations. Commun. Math. Phys., 248, 429-466, 2004.


\bibitem{Wang2021Layer-Splitting}
T. Wang, H. Chen, A. Zhou and Y. Zhou, Layer-Splitting methods for time-dependent Schr\"{o}dinger equations of incommensurate systems. Commun. Comput. Phys., 30, 1474-1498, 2021.
	

\bibitem{frigo2005design}
M. Frigo and S. Johnson, The design and implementation of FFTW3. Proc. IEEE, 93, 216-231, 2005.
	
	
\end{thebibliography}
\end{document}